\documentclass[a4paper,12pt]{article}

\usepackage{amssymb}
\usepackage{amsthm}
\usepackage{amsmath}
\usepackage{color}
\usepackage[tight]{subfigure}
\usepackage{microtype}

\usepackage{url}
\usepackage[pdftex]{graphicx}
\usepackage{rotating}
\usepackage{booktabs}

\newtheorem{theorem}{Theorem}
\newtheorem{lemma}{Lemma}
\newtheorem{corollary}{Corollary}
\newtheorem{definition}{Definition}
\newtheorem{remark}{Remark}
\newtheorem{assumption}{Assumption}
\newcounter{problem}

\usepackage[subnum]{cases}

\newcommand{\bs}{\boldsymbol}
\newcommand{\mr}[1]{\mathrm{#1}}
\newcommand{\E}{\mathrm{\bf E}}

\newcommand{\Rb}{\mathbb{R}}

\newcommand{\argmin}{\operatornamewithlimits{arg\,min}}

\renewcommand{\eqref}[1]{Eq.~\mbox{(\ref{#1})}}

\makeatletter

\newcommand{\Rmnum}[1]{\expandafter\@slowromancap\romannumeral #1@}
\makeatother

\DeclareMathAlphabet{\mathpzc}{OT1}{pzc}{m}{it}
\DeclareMathAlphabet{\mathcalligra}{T1}{calligra}{m}{n}

\newcommand{\oracle}{\mathcal{O}}

\makeatletter
\def\url@leostyle{%
  \@ifundefined{selectfont}{\def\UrlFont{\sf}}{\def\UrlFont{\small\ttfamily}}}
\makeatother
\usepackage[belowskip = 10.5 pt, footnotesize, bf]{caption}

\author{Milan~Korda and Colin N. Jones%
\thanks{Milan~Korda and Colin~N.~Jones are with the Laboratoire d'Automatique , \'Ecole Polytechnique F\'ed\'erale de Lausanne, Switzerland. {\tt\footnotesize \{milan.korda, colin.jones\}@epfl.ch}.}
}

\title{\bf Stability and Performance Verification of Optimization-based Controllers}

\begin{document}

\maketitle


\markboth{IEEE Transactions on Automatic Control, January 2015, Submitted}%
{Shell \MakeLowercase{\textit{et al.}}: Bare Demo of IEEEtran.cls for Journals}

\begin{abstract}
This paper presents a method to verify closed-loop properties of optimization-based controllers for deterministic and stochastic constrained polynomial discrete-time dynamical systems. The closed-loop properties amenable to the proposed technique include global and local stability, performance with respect to a given cost function (both in a deterministic and stochastic setting) and the $\mathcal{L}_2$ gain. The method applies to a wide range of practical control problems: For instance, a dynamical controller (e.g., a PID) plus input saturation, model predictive control with state estimation, inexact model and soft constraints, or a general optimization-based controller where the underlying problem is solved with a fixed number of iterations of a first-order method are all amenable to the proposed approach.

The approach is based on the observation that the control input generated by an optimization-based controller satisfies the associated Karush-Kuhn-Tucker (KKT) conditions which, provided all data is polynomial, are a system of polynomial equalities and inequalities. The closed-loop properties can then be analyzed using sum-of-squares (SOS) programming.
\end{abstract}

\begin{center}
\end{center}

\maketitle

\section{Introduction}


This paper presents a computational approach to analyze closed-loop properties of optimization-based controllers for constrained polynomial discrete-time dynamical systems. Throughout the paper we assume that we are given an optimization-based controller that at each time instance generates a control input by solving an optimization problem parametrized by a function of the past measurements of the controlled system's output, and we ask about closed-loop properties of this interconnection. This setting encompasses a wide range of control problems including the control of a polynomial dynamical system by a linear controller (e.g., a PID) with an input saturation, output feedback model predictive control with inexact model and soft constraints, or a general optimization-based controller where the underlying problem is solved approximately with a fixed number of iterations of a first-order\footnote{By a first order optimization method we mean a method using only function value and gradient information, e.g., the projected gradient method (see Section~\ref{sec:fom}).} optimization method. Importantly, the method verifies all KKT points; hence it can be used to verify closed-loop properties of optimization-based controllers where the underlying, possibly nonconvex, optimization problem is solved with a local method with guaranteed convergence to a KKT point only.

The closed-loop properties possible to analyze by the approach include: global stability and stability on a given subset, performance with respect to a discounted infinite-horizon cost (where we provide polynomial upper and lower bounds on the cost attained by the controller over a given set of initial conditions, both in a deterministic and a stochastic setting), the $\mathcal{L}_2$ gain from a given disturbance input to a given performance output (where we provide a numerical upper bound).


The main idea behind the presented approach is the observation that the KKT system associated to an optimization problem with polynomial data is a system of polynomial equalities and inequalities. Consequently, provided that suitable constraint qualification conditions hold (see, e.g.,~\cite{peterson1973}), the solution of this optimization problem satisfies a system of polynomial equalities and inequalities. Hence, the closed-loop evolution of a polynomial dynamical system controlled by an optimization-based controller solving at each time step an optimization problem with polynomial data can be seen as a difference inclusion where the successor state lies in a set defined by polynomial equalities and inequalities. This difference inclusion is then analyzed using sum-of-squares (SOS) techniques (see, e.g., \cite{lasserreBook,parrilo} for introduction to SOS programming).

The approach is based on the observation of Primbs~\cite{primbs} who noticed that the KKT system of a constrained linear-quadratic optimization problem is a set of polynomial equalities and inequalities and used the S-procedure (see, e.g., \cite{boyd_LMI}) to derive sufficient linear matrix inequality (LMI) conditions for a given linear MPC controller to be stabilizing. In this work we significantly extend the approach in terms of both the range of closed-loop properties analyzed and the range of practical problems amenable to the method. Indeed, our approach is applicable to general polynomial dynamical systems, both deterministic and stochastic, and allows the analysis not only of stability but also of various performance measures. The approach is not only applicable to an MPC controller with linear dynamics and a quadratic cost function as in~\cite{primbs} but also to a general optimization-based controller, where the optimization problem may not be solved exactly, encompassing all the above-mentioned control problems.

This work is a continuation of~\cite{korda_ACC_verif} where the approach was used to analyze the stability of optimization-based controllers where the optimization problem is solved approximately by a fixed number of iterations of a first order method. The results of~\cite{korda_ACC_verif} are summarized in Section~\ref{sec:fom} of this paper as one of the examples that fit in the presented framework.

The paper is organized as follows. Section~\ref{sec:SOSprog} gives a brief introduction to SOS programming. Section~\ref{sec:probSetup} states the problem to be solved. Section~\ref{sec:examples} presents a number of examples amenable to the proposed method. Section~\ref{sec:closedLoopAnalysis} presents the main verification results: Section~\ref{sec:stabGlobal} on global stability, Section~\ref{sec:stabLocal} on stability on a given subset, Section~\ref{sec:perfDet} on performance analysis in a deterministic setting, Section~\ref{sec:perfStoch} on performance analysis in a stochastic setting and Section~\ref{sec:perfRob} on the analysis of the $\mathcal{L}_2$ gain in a robust setting. Computational aspects are discussed in Section~\ref{sec:compAspects}. Numerical examples are in Section~\ref{sec:numEx} and some proofs are collected in the Appendix.

\section{Sum-of-squares programming}\label{sec:SOSprog}
Throughout the paper we will rely on sum-of-squares (SOS) programming, which allows us to optimize, in a convex way, over polynomials with nonnegativity constraints imposed over a set defined by polynomial equalities and inequalities (see, e.g., \cite{parrilo,lasserreBook} for more details on SOS programming). In particular we will often encounter optimization problems with constraints of the form 
\begin{equation}\label{eq:nonnegGen}
\mathcal{L}V (x) \ge 0 \quad \forall\, x \in \mathbf{K},
\end{equation}
where $V : \Rb^n \to \Rb $ is a polynomial, $\mathcal{L}$ a linear operator mapping polynomials to polynomials (e.g., a simple addition or a composition with a fixed function) and
\[
\mathbf{K} = \{x\in \mathbb{R}^n \mid g(x) \ge 0, \; h(x) = 0\},
\]
where the functions $g: \Rb^n \to \Rb^{n_g}$ and $h: \Rb^n \to \Rb^{n_h}$ are vector polynomials (i.e., each component is a polynomial). A \emph{sufficient} condition for~(\ref{eq:nonnegGen}) to be satisfied is
\begin{equation}\label{eq:sosGen}
\mathcal{L}V = \sigma_0 + \sum_{i=1}^{n_g}\sigma_i g_i + \sum_{i=1}^{n_h}p_i h_i,
\end{equation}
where $\sigma_0$ and $\sigma_i$, $i = 1,\ldots,n_g$, are SOS polynomials and $p_i$, $i=1,\ldots, n_h$, are arbitrary polynomials. A polynomial $\sigma$ is SOS if it can be written as
\begin{equation}\label{eq:sos_eq}
\sigma(x) = \beta(x)^\top \mathcal{Q}\beta(x), \quad \mathcal{Q} \succeq 0,
\end{equation}
where $\beta(x)$ is a vector of polynomials and $\mathcal{Q} \succeq 0$ signifies that $\mathcal{Q}$ is a positive semidefinite matrix. The condition~(\ref{eq:sos_eq}) trivially implies that $\sigma(x) \ge 0$ for all $x \in \Rb^n$. Importantly, the condition~(\ref{eq:sos_eq}) translates to a set of linear constraints and a positive semidefiniteness constraint and therefore is equivalent to a semidefinite programming (SDP) feasibility problem. In addition, the constraint~(\ref{eq:sosGen}) is affine in the coefficients of $V$, $\sigma$ and $p$; therefore~(\ref{eq:sosGen}) also translates to an SDP feasibility problem and, crucially, it is possible to \emph{optimize} over the coefficients of $V$ (as long as they are affinely parametrized in the decision variables) subject to the constraint~(\ref{eq:sosGen}) using semidefinite programming.

In the rest of the paper when we say that a constraint of the form~(\ref{eq:nonnegGen}) is replaced by sufficient SOS constraints, then we mean that~(\ref{eq:nonnegGen}) is replaced with~(\ref{eq:sosGen}).

In addition we will often encounter optimization problems with objective functions of the form
\begin{equation}\label{eq:objGen}
\mr{min} / \mr{max}\; \int_\mathbf{X} V(x)\, dx,
\end{equation}
where $V$ is a polynomial and $\mathbf{X}$ a simple set (e.g., a box). The objective function is linear in the coefficients of the polynomials $V$. Indeed, expressing $V(x) = \sum_{i=1}^{n_\beta} v_i \beta_i(x)$, where $(\beta_i)_{i=1}^{n_\beta}$ are fixed polynomial basis functions and $(v_i)_{i=1}^{n_\beta}$ the corresponding coefficients, we have
\[
\int_\mathbf{X} V(x)\,dx = \sum_{i=1}^{n_\beta} v_i \int_\mathbf{X} \beta_i(x)\, dx = \sum_{i=1}^{n_\beta} v_i m_i,
\]
where the moments $m_i := \int_\mathbf{X} \beta_i(x)\, dx$ can be precomputed (in a closed form for simple sets $\mathbf{X}$). We see that the objective~(\ref{eq:objGen}) is linear in the coefficients $(v_i)_{i=1}^{n_\beta}$ and hence optimization problems with objective~(\ref{eq:objGen}) subject to the constraint~(\ref{eq:nonnegGen}) enforced via the sufficient constraint~(\ref{eq:sosGen}) translate to an SDP.



\section{Problem statement}\label{sec:probSetup}
We consider the nonlinear discrete-time dynamical system
\begin{subequations}\label{eq:sys}
\begin{align}
x^+ &= f_x(x,u), \\
y &= f_y(x),
\end{align}
\end{subequations}
where $x \in \Rb^{n_x}$ is the state, $u \in \Rb^{n_u}$ the control input, $y \in \Rb^{n_y}$ the output, $x^+ \in \Rb^{n_x}$ the successor state, $f_x : \Rb^{n_x}\times \Rb^{n_u} \to \Rb^{n_x}$ a transition mapping and $f_y: \Rb^{n_x} \to \Rb^{n_y}$ an output mapping. We assume that each component of $f_x$ and $f_y$ is a multivariate polynomial in $(x,u)$ and $x$, respectively.

We assume that the system is controlled by a given set-valued controller
\begin{equation}\label{eq:cont}
u \in \kappa(\mathbf{K}_s),
\end{equation}
where $\kappa: \Rb^{n_\theta}\to \Rb^{n_u}$ is polynomial and
\begin{equation}\label{eq:Kz}
\mathbf{K}_s := \{ \theta \in \Rb^{n_\theta} \! \mid \!  \exists \lambda\in \mathbb{R}^{n_\lambda}\; \mr{s.t.} \; g(s,\theta,\lambda) \ge 0, h(s,\theta,\lambda   ) = 0\},
\end{equation}
where each component of the vector-valued functions $g:\Rb^{n_s}\times \Rb^{n_\theta} \times \Rb^{n_\lambda} \to \Rb^{n_g} $ and $h:\Rb^{n_s}\times \Rb^{n_\theta} \times \Rb^{n_\lambda} \to \Rb^{n_h} $ is a polynomial in the variables $(s,\theta,\lambda)$. The set $\mathbf{K}_s$ is parametrized by the output of a dynamical system
\begin{subequations}\label{eq:zsys}
\begin{align}
z^+ &= f_z(z,y), \\
s &= f_s(z,y),
\end{align}
\end{subequations}
where  $f_z: \Rb^{n_z} \times \Rb^{n_y} \to  \Rb^{n_z}$ and $f_s: \Rb^{n_z} \times \Rb^{n_y} \to  \Rb^{n_s}$ are polynomial. The problem setup is depicted in Figure~\ref{fig:probSetup}. In the rest of the paper we develop a method to analyze the closed-loop stability and performance of this interconnection. Before doing that we present several examples which fall into the presented framework.

\begin{figure}[htb]
	\begin{picture}(50,45)
	\put(35,3){\includegraphics[width=80mm]{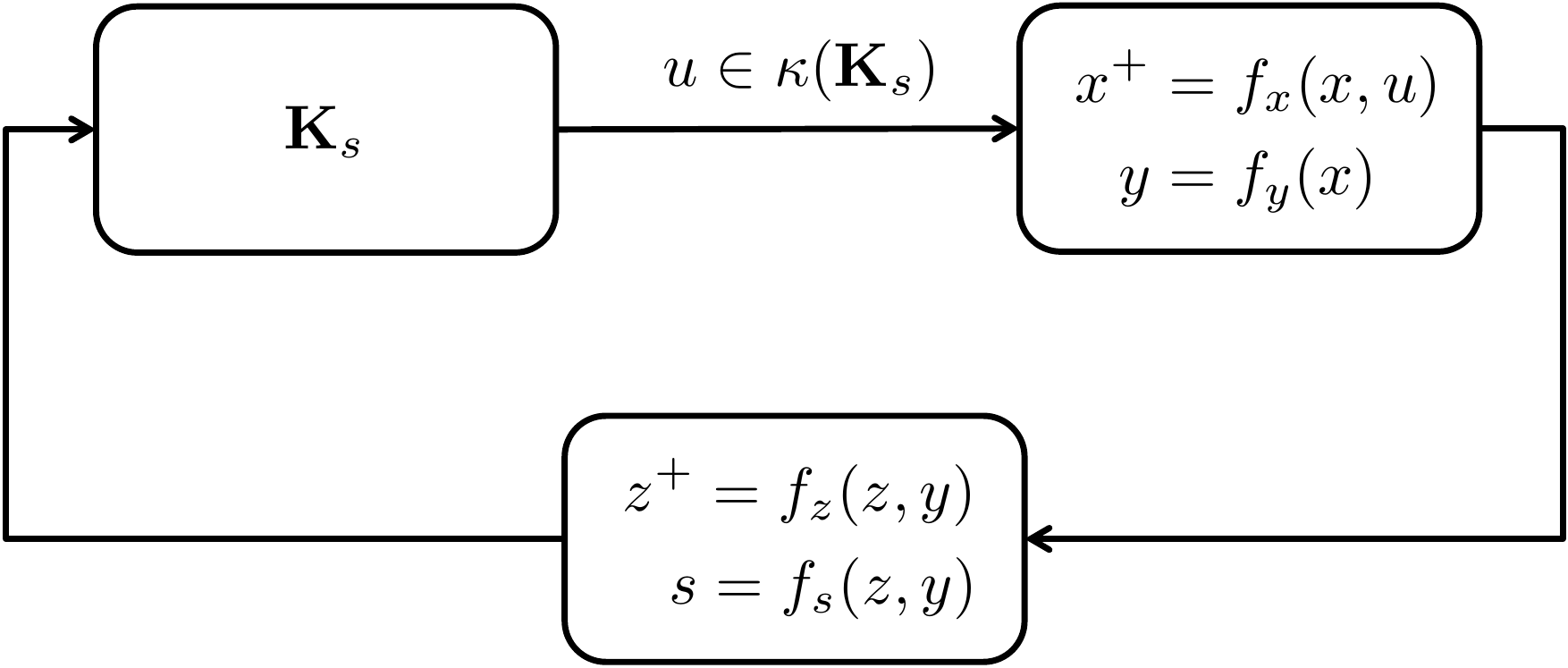}}
	\end{picture}
	\caption{\small Control scheme}
	\label{fig:probSetup}
\end{figure}

\section{Examples}\label{sec:examples}
The framework considered allows for the analysis of a large number of practical control scenarios. The common idea is to write the control input $u$ as the output of an optimization problem with polynomial data parametrized by the state of the dynamical system~(\ref{eq:zsys}). The control input $u$ then belongs to the associated KKT system (provided that mild regularity conditions are satisfied) which is of the form~(\ref{eq:Kz}).

\subsection{Polynomial dynamical controller + input saturation}\label{sec:inputSat}
Any polynomial dynamical controller (e.g., a PID controller) plus an input saturation can be written in the presented form provided that the input constraint set is defined by finitely many polynomial inequalities satisfying mild constraint qualification conditions (see, e.g., \cite{peterson1973}). Indeed, regarding $z$ as the state and $s$ as the output of the controller and generating $u$ according to $u \in \mathrm{proj}_{U}(s)$, where
\begin{equation}\label{eq:proj}
\mathrm{proj}_{U}(s) = \argmin_{\theta \in U}  \frac{1}{2}\|  \theta - s \|^2_2 
\end{equation}
is the set of Euclidean projections of $s$ on the constraint set $U$ (there can be multiple projections since $U$ is in general nonconvex). Assuming that the input constraint set is of the form
\begin{equation}\label{eq:inputConst}
U = \{ v \in \Rb^{n_u} \mid g_U(v) \ge 0\},
\end{equation}
where $g_U:\Rb^{n_u} \to \Rb^{n_{g_U}}$ has polynomial entries, the KKT conditions associated to the optimization problem~(\ref{eq:proj}) read
\begin{subequations}\label{eq:KKTsys_inputSatur}
\begin{align}
\theta- s -  \nabla g_U(\theta)\lambda &= 0 \\
\lambda ^\top g_U(\theta) &= 0 \\
\lambda &\ge 0 \\
g_U(\theta) &\ge 0,
\end{align}
\end{subequations}
where $\lambda \in \Rb^{n_u}$ is the vector of Lagrange multipliers associated with the constraints defining $U$ and $\nabla g_U$ is the transpose of the Jacobian of $g_U$ (i.e., $[\nabla g_U]_{i,j} = \frac{\partial [g_U]_j}{\partial x_i}$). Assuming that constraint qualification conditions hold such that any minimizer of~(\ref{eq:proj}) satisfies the KKT conditions~(\ref{eq:KKTsys_inputSatur}) we conclude that
\[
u  \in \kappa(\mathbf{K}_s)
\]
with $\kappa$ being the identity (i.e., $\kappa(\theta) = \theta$),
 \[
h(s,\theta,\lambda) =
\begin{bmatrix}
 \theta- s -  \nabla g_U(\theta) \lambda  \\  \lambda^\top g_U(\theta) 
\end{bmatrix}
\]
and
\[
g(s,\theta,\lambda) = \begin{bmatrix} 
\lambda \\ g_U(\theta)
\end{bmatrix},
\]
where $h$ and $g$ are polynomials in $(s,\theta,\lambda)$ as required.

Note that the description of the input constraint set~(\ref{eq:inputConst}) is not unique. For example, if the input constraint set is $[-1,1]$, then the function $g_U$ can be
\begin{equation}\label{eq:desc_gU1}
g_U(\theta) = (1-\theta)(1+\theta)
\end{equation}
or
\begin{equation}\label{eq:desc_gU2}
g_U(\theta) = \begin{bmatrix} 1-\theta \\ 1+\theta  \end{bmatrix}
\end{equation}
or any odd powers of the above. Depending on the particular description, the constraint qualification conditions may or may not hold. It is therefore important to choose a suitable description of $U$ which is both simple and such that the constraint qualification conditions hold. We remark that in the case of $U = [-1,1]$ both~(\ref{eq:desc_gU1}) and (\ref{eq:desc_gU2}) satisfy these requirements.

Note also that the KKT system~(\ref{eq:KKTsys_inputSatur}) may be satisfied by points which are not global minimizers of~(\ref{eq:proj}) if the set $U$ is nonconvex; this is an artefact of the presented method and cannot be avoided within the presented framework. We note, however, that the input constraint set $U$ is in most practical cases convex (note that the concavity of the components of $g_U$ is not required; what matters is the convexity of the set $U$ defined by $g_U$~\cite{lasserre_convexRepresentation}).

\subsection{Output feedback nonlinear MPC with model mismatch and soft constraints}
This example shows how to model within the presented framework a nonlinear MPC controller with state estimation, a model mismatch\footnote{In this example we assume that we know the true model of the system but in the MPC controller we intentionally use a different model (e.g., we use a linearized or otherwise simplified model for the sake of computation speed); the true model is used only to \emph{verify} closed-loop properties of the true model controlled by the MPC controller. See Section~\ref{sec:perfRob} for the case where the true model is not known exactly even for the verification purposes and the model mismatch is captured by an exogenous disturbance.}, soft constraints and no a priori stability guarantees (enforced, e.g., using a terminal penalty and/or terminal set) and possibly only locally optimal solutions delivered by the optimization algorithm. In this case the system~(\ref{eq:zsys}) is an estimator of the state of the dynamical system~(\ref{eq:sys}) and in each time step the following optimization problem is solved
\begin{equation}
\label{opt:mpc}
\begin{array}{ll}
\underset{\boldsymbol{ \hat{u} }, \boldsymbol{ \hat{x},\boldsymbol{\varepsilon} } }{\mbox{minimize}} & l_\mr{s}(\boldsymbol{\varepsilon})  + \sum_{i=0}^{N-1} l_i(s,\hat{x}_i,\hat{u}_i) + l_N(s,x_N)\\
\mbox{subject to} & \hat{x}_{i+1} = \hat{f}(\hat{x}_i,\hat{u}_i), \quad i = 0,\ldots, N-1 \\
& a(s,\boldsymbol{\hat{x}}, \boldsymbol{\hat{u}},\boldsymbol \varepsilon )\ge 0 \\
& b(s,\boldsymbol{\hat{x}}, \boldsymbol{\hat{u}}) = 0 ,
\end{array}
\end{equation}
where $\boldsymbol{\hat{x}} = (\hat{x}_0,\ldots, \hat{x}_N)$, $\boldsymbol{\hat{u}} = (\hat{u}_0,\ldots, \hat{u}_{N-1})$, $\boldsymbol{\varepsilon} = (\varepsilon_1,\ldots,\varepsilon_{n_\varepsilon})$ are slack variables for the inequality constraints, $ \hat{f}$ is a polynomial model of the true transition mapping $f_x$, $l_\mr{s}$ is a polynomial penalty for violations of the inequality constraints, $l_i$, $i=0,\ldots,N$, are polynomial stage costs and $a$ and $b$ (vector) polynomial constraints parametrized by the state estimate $s$ produced by~(\ref{eq:zsys}). If the dimension of the state estimate $s$ is equal to the dimension of the state of the model $\hat{x}$ (which we do not require), then most MPC formulations will impose $\hat{x}_0 = s$, which is encoded by making one of the components of $b$ equal to $\hat{x}_0 - s$. The formulation~(\ref{opt:mpc}) is, however, not restricted to this scenario and allows arbitrary dependence of the constraints (along the whole prediction horizon) on the state estimate $s$. The control input applied to the system is then $u = \hat{u}_0^\star$, where $\hat{u}_0^\star$ is the first component of any vector $\boldsymbol{\hat{u}^\star}$ optimal in~(\ref{opt:mpc}).

The KKT system associated to~(\ref{opt:mpc}) reads
\begin{subequations}\label{eq:KKTmpc}
\begin{align}
 \nabla_{\hat{\boldsymbol{x}},\hat{\boldsymbol{u}},\bs{\varepsilon}} \mathcal{L}(s,\bs{\hat{x}},\bs{\hat{u}},\bs{\varepsilon},\lambda) & = 0\\
\lambda_a ^\top a(s,\boldsymbol{\hat{x}}, \boldsymbol{\hat{u}},\bs{\varepsilon}) &= 0 \\
b(s,\boldsymbol{\hat{x}}, \boldsymbol{\hat{u}}) &= 0\\
\hat{x}_{i+1} - \hat{f}(\hat{x}_i,\hat{u}_i) & = 0, \quad i = 0,\ldots, N-1
\\
\lambda_a &\ge 0 \\
a(s,\boldsymbol{\hat{x}}, \boldsymbol{\hat{u}}, \bs{\varepsilon}) &\ge 0, 
\end{align}
\end{subequations}
where $\lambda := (\lambda_a,\lambda_b,\lambda_{\hat{f}}^0,\ldots,\lambda_{\hat{f}}^{N-1})$ and
\begin{align*}
 \mathcal{L}(s,\hat{x},\hat{u},\bs{\varepsilon},\lambda) & \!:= l_\mr{s}(\bs{\varepsilon})+ \!\! \sum_{i=0}^{N-1}   l_i(s,\hat{x}_i,\hat{u}_i) - \lambda_{a}^\top( a(s,\boldsymbol{\hat{x}}, \boldsymbol{\hat{u}} , \bs{\varepsilon})  )  \\ 
& \hspace{-1.5cm} + l_N(s,x_N) +\lambda_{b}^\top b(s,\boldsymbol{\hat{x}}, \boldsymbol{\hat{u}}) + \sum_{i=0}^{N-1}\lambda_{\hat{f}}^{i}(\hat{x}_{i+1} - \hat{f}(\hat{x}_i,\hat{u}_i)) 
\end{align*}
is the Lagrangian of~(\ref{opt:mpc}). The KKT system~(\ref{eq:KKTmpc}) is a system of polynomial equalities and inequalities. Consequently, setting 
\[
\theta := (\boldsymbol{\hat{u}},\boldsymbol{\hat{x}},\bs{\varepsilon}), \quad\kappa(\theta) = \hat{u}_0
\]
and assuming that constraint qualification conditions hold such that every optimal solution to~(\ref{opt:mpc}) satisfies the KKT condition~(\ref{eq:KKTmpc}), there exist polynomial functions $h(s,\theta,\lambda) $ and $g(s,\theta,\lambda) $ such that $\hat{u}_0^\star \in \kappa (\mathbf{K}_s)$ for every $\hat{u}_0^\star$ optimal in~(\ref{opt:mpc}).

\begin{remark}\label{rem:localSolutions} Let us mention that, provided suitable constraint qualification conditions hold, not only every globally optimal $\hat{u}_0^\star$ will satisfy the KKT system but also every \emph{locally optimal} solution to~(\ref{opt:mpc}) and every critical point of~(\ref{opt:mpc}) will; hence the proposed method can be used to verify stability and performance properties even if only local solutions to~(\ref{opt:mpc}) are delivered by the optimization algorithm.
\end{remark}

\begin{remark}\label{rem:inexactSol}
Note that the situation where the optimization problem~(\ref{opt:mpc}) is not solved exactly can be handled as well. One way to do so is to include an auxiliary variable $\delta$ capturing the inaccuracy in the solution, either in the satisfaction of the KKT system~(\ref{eq:KKTmpc}) or directly as an error on the delivered control action $\hat{u}_0$ (e.g., defining $\kappa(\theta) = \hat{u}_0(1+\delta)$ with $\theta = (\boldsymbol{\hat{u}},\boldsymbol{\hat{x}},\bs{\varepsilon},\delta)$) and imposing $|\delta| \le \Delta $, where $\Delta > 0$ is a known bound on the solution accuracy. If the solution inaccuracy is due to a premature termination of a first-order optimization method used to solve~(\ref{opt:mpc}), a more refined analysis can be carried out within the presented framework; this is detailed in Section~\ref{sec:fom}.
\end{remark}

\subsection{General optimization-based controller}\label{sec:generalOBC}
Clearly, there was nothing specific about the MPC structure of the optimization problem solved in the previous example and therefore the presented framework can be used to analyze arbitrary optimization-based controllers which at each time step solve an optimization problem parametrized by the output of the dynamical system~(\ref{eq:zsys}):
\begin{equation}
\label{opt:gen}
\begin{array}{ll}
\underset{\theta\in \Rb^{n_\theta}}{\mbox{minimize}} & J(s,\theta)\\
\mbox{subject to} & a(s,\theta) \ge 0 \\
& b(s,\theta) = 0 ,
\end{array}
\end{equation}
where $J$ is a polynomial and $a$ and $b$ polynomial vectors. The associated KKT system reads
\begin{subequations}\label{eq:KKTgen}
\begin{align}
 \nabla_\theta J(s,\theta)  - \nabla_\theta a(s,\theta) \lambda_{a} + \nabla_\theta b(s,\theta) \lambda_{b} &= 0 \\
\lambda_a ^\top a(s,\theta) &= 0 \\
\lambda_a &\ge 0 \\
a(s,\theta) &\ge 0 \\
b(s,\theta) &= 0,
\end{align}
\end{subequations}
which is a system of polynomial equalities and inequalities in $(s,\theta,\lambda)$, where  $\lambda = (\lambda_a,\lambda_b)$ and hence can be treated within the given framework. In particular the functions $h$ and $g$ defining the set $\mathbf{K}_s$ read
\[
h(s,\theta,\lambda) \! = \!
\begin{bmatrix} 
\nabla_\theta J(s,\theta)  - \nabla_\theta a(s,\theta) \lambda_{a} + \nabla_\theta b(s,\theta) \lambda_{b} \\ 
\lambda_a ^\top a(s,\theta)\\
b(s,\theta)
\end{bmatrix}
\]
and
\[
g(s,\theta,\lambda) = 
\begin{bmatrix} 
\lambda_a\\
a(s,\theta)
\end{bmatrix}.
\]

See Remark~\ref{rem:inexactSol} for the situation where the problem~(\ref{opt:gen}) is not solved exactly.

\subsection{Optimization-based controller solved using a fixed number of iterations of a first order method}\label{sec:fom}
The presented framework can also handle the situation where the optimization problem~(\ref{opt:gen}) is solved using an iterative optimization method, each step of which is either an optimization problem or a polynomial mapping. This scenario was elaborated on in detail in~\cite{korda_ACC_verif}, where it was shown that the vast majority of first order optimization methods fall into this category. Here we present the basic idea on one of the simplest optimization algorithms, the projected gradient method. When applied to problem~(\ref{opt:gen}), the iterates of the projected gradient method are given by
\begin{equation}\label{eq:gradStep}
\theta_{k+1} \in  \mr{proj}_{\mathcal{S}}(\theta_k - \eta \nabla_\theta J(s,\theta_k) ),
\end{equation}
where $\mr{proj}_{\mathcal{S}}(\cdot)$ denotes the set of Euclidean projections on the constraint set
\[
\mathcal{S} = \{\theta\in\mathbb{R}^{n_\theta} \mid a(s,\theta) \ge 0, \: b(s,\theta) = 0\}
\] 
and $\eta > 0$ is a step size. The update formula~(\ref{eq:gradStep}) can be decomposed into two steps: step in the direction of the negative gradient and projection on the constraint set. The first step is a polynomial mapping and the second step is an optimization problem. Indeed, equation~(\ref{eq:gradStep}) can be equivalently written as
\begin{equation}\label{opt_grad}
\begin{array}{rclll}
 \theta_{k+1} &\in& \underset{\;\,\theta\in\mathbb{R}^{n_\theta}}{\mathrm{arg\,min}} &  \frac{1}{2}\| \theta - (\theta_k - \eta\nabla_\theta{J(s,\theta_k)}) \|^2_2 \vspace{1.2mm}\\
&&\hspace{0.6cm} \mathrm{s.t.} & a(s,\theta) \ge 0 \vspace{0.4mm}\\
&&& b(s,\theta) = 0.
\end{array}
\end{equation}
For each $k \in\{ 0,1,\ldots\}$, the KKT system associated to~(\ref{opt_grad}) reads
\begin{subequations}\label{eq:KKT_firstOrder}
\begin{align}
 \nonumber \theta_{k+1} - (\theta_k - \eta\nabla_\theta{J(s,\theta_k)}) & - \nabla_\theta\, a(s,\theta_{k+1})  \lambda_a^{k+1} \\  + \nabla_\theta\, b(s,\theta_{k+1})  \lambda_b^{k+1} & = 0 \\
 b(s,\theta_{k+1}) &= 0 \label{eq:KKT_FO_stat}\\
 a(s,\theta_{k+1})^\top \lambda_a^{k+1} & = 0\\
 a(s,\theta_{k+1})  & \ge 0\\
 \lambda_a^{k+1}  & \ge 0, 
\end{align}
\end{subequations}
which is a system of polynomial equalities and inequalities. Note in particular the coupling between $\theta_k$ and $\theta_{k+1}$ in equation~(\ref{eq:KKT_FO_stat}). Assuming we apply $M$ steps of the projected gradient method, the last iterated $\theta_M$ is therefore characterized by $M$ coupled KKT systems of the form~(\ref{eq:KKT_firstOrder}), which is a system of polynomial equalities and inequalities as required by the proposed method.

Other optimization methods, in particular most of the first order methods (e.g., fast gradient method~\cite{nesterov_book}, AMA~\cite{tseng_ama}, ADMM~\cite{gabay1976_admm} and their accelerated versions~\cite{goldstein2014fast}), including local non-convex methods (e.g., \cite{jeanNonconvex}), and some of the second order methods (e.g., the interior point method with exact line search) are readily formulated in this framework as well; see~\cite{korda_ACC_verif} for more details on first-order methods.

\section{Closed-loop analysis}\label{sec:closedLoopAnalysis}
In this section we describe a method to analyze closed-loop properties of the interconnection depicted in Figure~\ref{fig:probSetup} and described in Section~\ref{sec:probSetup}. First, notice that the closed-loop evolution is governed by the difference inclusion
\begin{subequations}\label{eq:difIncl}
\begin{align}
x^+ &\in f_x(x,\kappa(\mathbf{K}_s)), \\
z^+ &= f_z(z,f_y(x)).
\end{align}
\end{subequations}
Since all problem data is polynomial and the set $\mathbf{K}_s$ is defined by polynomial equalities and inequalities it is possible to analyze stability and performance using sum-of-squares (SOS) programming. This is detailed next.

We will use the following notation:
\begin{subequations}\label{eq:hatFunctions}
\begin{align}
\hat{h}(x,z,\theta,\lambda) & := h(f_s(z,f_y(x)),\theta,\lambda  )\\
\hat{g}(x,z,\theta,\lambda) & := g(f_s(z,f_y(x)),\theta,\lambda  ).
\end{align}
\end{subequations}

For the rest of the paper we impose the following standing assumption:
\begin{assumption}\label{as:nonempty}
The set $\mathbf{K}_s$ is nonempty for all $s \in \Rb^{n_s}$.
\end{assumption}
Assumption~\ref{as:nonempty} implies that the control input~(\ref{eq:cont}) is well defined for all $s \in \Rb^{n_s}$.

\subsection{Stability analysis -- global}\label{sec:stabGlobal}
A sufficient condition for the state $x$ of the difference inclusion~(\ref{eq:difIncl}) to be stable is the existence of a function $V$ satisfying
\begin{subequations}\label{eq:Lyap}
\begin{align}
V(x^+,z^+,\theta^+,\lambda^+) - V(x,z,\theta,\lambda) & \le -\| x\|^2_2 \label{eq:LyapDec} \\
V(x,z,\theta,\lambda) \ge \|x \|^2_2 \label{eq:LyapNonNeg}
\end{align}
\end{subequations}
for all
\[ 
(x,z,\theta,\lambda,x^+,z^+,\theta^+,\lambda^+) \in \mathbf{K},
\]
where 
\begin{align} \label{eq:K}
&\mathbf{K} = \{ (x,z,\theta,\lambda,x^+,z^+,\theta^+,\lambda^+) \mid x^+ = f_x(x,\kappa(\theta)), \nonumber \\
&\hat{h}(x,z,\theta,\lambda) = 0, \; \hat{g}(x,z,\theta,\lambda) \ge 0, \; \hat{h}(x^+,z^+,\theta^+,\lambda^+) = 0, \nonumber \\
&\hat{g}(x^+,z^+,\theta^+,\lambda^+) \ge 0, \; z^+ = f_z(z,f_y(x))
  \}.
\end{align}

These equations require that a Lyapunov function $V$ exists which decreases on the basic semialgebraic set $\mathbf{K}$ implicitly characterizing the closed-loop evolution~(\ref{eq:difIncl}). Therefore, we can tractably seek a Lyapunov function for system~~(\ref{eq:difIncl}) by restricting $V$ to be a polynomial of a pre-defined degree and replacing the inequalities~(\ref{eq:Lyap}) by sufficient SOS conditions according to Section~\ref{sec:SOSprog}. For better understanding we detail this replacement here and refer to the general treatment in Section~\ref{sec:SOSprog} in the rest of the paper.

 Setting
\[
 \xi := (x,z,\theta,\lambda,x^+,z^+,\theta^+,\lambda^+),
\]
 these SOS conditions read
\begin{subequations}\label{eq:sos_eqs}
\begin{align}\label{eq:sos_decrease}
V(x,z,\theta,\lambda) - V(x^+,z^+,\theta^+,\lambda^+)  - \| x \| ^2_2 = \\\nonumber   &\hspace{-6cm}\sigma_0(\xi) + \sigma_1(\xi)^\top \hat{g}(x,z,\theta,\lambda) + \sigma_2(\xi )^\top \hat{g}(x^+,z^+,\theta^+,\lambda^+)\\\nonumber & \hspace{-6cm}+ p_1(\xi)^\top \hat{h}(x,z,\theta,\lambda) + p_2(\xi)^\top  \hat{h}(x^+,z^+,\theta^+,\lambda^+)\\ \nonumber &\hspace{-6cm} + p_3(\xi)(x^+ - f_x(x,\kappa(\theta))
+ p_4(\xi) (z^+ - f_z(z,f_y(x)), \\ 
\label{eq:sos_nonneg}
&\hspace{-6.4cm} V(x,z,\theta, \lambda) - \| x \| ^2_2  = \\\nonumber   &\hspace{-6cm}\bar{\sigma}_0(\xi) + \bar{\sigma}_1(\xi)^\top \hat{g}(x,z,\theta,\lambda) + \bar{p}_1(\xi)^\top \hat{h}(x,z,\theta,\lambda), 
\end{align}
\end{subequations}
where $\sigma_i(\xi)$ and $\bar{\sigma}_i(\xi)$ are SOS multipliers and $p_i(\xi)$ and $\bar{p}_i(\xi)$ polynomial multipliers of compatible dimensions and pre-specified degrees (selection of the degrees is discussed in Section~\ref{sec:compAspects}). The satisfaction of~(\ref{eq:sos_decrease}) implies the satisfaction of~(\ref{eq:LyapDec}) and the satisfaction of~(\ref{eq:sos_nonneg}) implies the satisfaction of~(\ref{eq:LyapNonNeg}) for all $\xi \in \mathbf{K}$; this follows readily since $\sigma_i$ and $\bar{\sigma}_i$ are globally nonnegative, $\hat{g}$ is nonnegative on $\bf K$ and $\hat{h}$ is zero on $\bf K$ (see Section~\ref{sec:SOSprog}).

\begin{remark}\label{rem:compos}
Note that instead of including the equalities $x^+ - f_x(x,\kappa(\theta))$ and $z^+ - f_z(z,f_y(x))$ in the description of $\bf K$ we could also directly substitute for $x^+$ and $z^+$. In general, direct substitution is preferred if the mappings $f_x$, $f_z$ and $f_y$ are of low degree, especially linear, in which case there is no increase in the degree of the composition of $V$ with $f_x$ or with $f_z$ and $f_y$. Otherwise, the formulation~(\ref{eq:sos_eqs}) is preferred.
\end{remark}

\begin{remark}\label{rem:other_vars}
Note that the Lyapunov function $V$ is allowed to depend not only on the state variables $x$ and $z$ but also on the variables $\theta$ and $\lambda$ describing the semialgebraic set $\mathbf{K}_s$. This added flexibility implies that a larger class of functions is spanned after projecting on the state variables $(x,z)$. In particular, if the variables $\theta$ and $\lambda$ are respectively the decision variables and Lagrange multipliers associated to a KKT system of an optimization problem, then the class of functions spanned includes the value function of this optimization problem which is typically \emph{not polynomial} (for example, this value function is \emph{piecewise} quadratic in the case of a standard linear MPC with quadratic cost and polytopic constraints).
\end{remark}

From the previous discussion we conclude that closed-loop stability of the state $x$ of~(\ref{eq:difIncl}) is implied by the feasibility of the following SOS problem:
\begin{equation}\label{opt:sos}
\begin{array}{rclll}
& \mathrm{find} & \displaystyle V,\sigma_0,\sigma_1,\sigma_2,p_1,p_2,p_3,p_4,\bar{\sigma}_0,\bar{\sigma}_1,\bar{p}_1 \vspace{1.2mm}\\
& \mathrm{s.t.} & (\ref{eq:sos_decrease}), (\ref{eq:sos_nonneg}) \vspace{0.4mm}\\
&& \sigma_0,\sigma_1,\sigma_2, \bar{\sigma}_0,\bar{\sigma}_1 & \hspace{-2.1cm}\text{SOS polynomials}\\
&& V,p_1, p_2,p_3,p_4, \bar{p}_1 &\hspace{-2.1cm}\text{arbitrary polynomials},
\end{array}
\end{equation}
where the decision variables are the coefficients of the polynomials \[(V,\sigma_0,\sigma_1,\sigma_2,p_1,p_2,p_3,p_4,\bar{\sigma}_0,\bar{\sigma}_1,\bar{p}_1).\]

The following theorem summarizes the results of this section.
\begin{theorem}
If the problem~(\ref{opt:sos}) is feasible, then the state $x$ of the closed-loop system~(\ref{eq:difIncl}) is globally asymptotically stable.
\end{theorem}

\subsection{Stability analysis -- on a given subset}\label{sec:stabLocal}
This section addresses the stability analysis on a given subset $\mathbf{X}$ of the state space $\Rb^{n_x} \times \Rb^{n_z}$ of the difference inclusion~(\ref{eq:difIncl}). The set $\mathbf{X}$ serves to restrict the search for a stability certificate to a given subset of the state space if global stability cannot be proven, as well as it can be used to encode physical constraints on the states of the original system $x$ and/or known relationships between the states $x$ and $z$ (e.g., if $z$ is an estimate of $x$ and a bound on the estimation error is known).

We assume that the set $\mathbf{X}$ is defined as
\begin{equation}\label{eq:Xdef}
\mathbf{X} := \{(x,z)\in \mathbb{R}^{n_x+n_z} \mid \psi_i(x,z) \ge 0, \; i=1,\ldots, n_\psi \}.
\end{equation}
where $\psi_i(\cdot)$ are polynomials. The Lyapunov conditions~(\ref{eq:Lyap}) are then enforced on the intersection of $\mathbf{X}$ with the set $\mathbf{K}$ (defined in~(\ref{eq:K})), i.e., on the set
\begin{align*}
\bar{\mathbf{K}}  := \{  & (x,z,\theta,\lambda,x^+,z^+,\theta^+,\lambda^+)   \mid \\ &   (x,z,\theta,\lambda,x^+,z^+,\theta^+,\lambda^+) \in \mathbf{K}, \;  (x,z)\in\mathbf{X} \},
\end{align*}
which is a set defined by finitely many polynomial equalities and inequalities. Hence the problem of stability verification on $\mathbf{X}$ leads to an SOS problem completely analogous to~(\ref{opt:sos}).



However, the pitfall here is that the satisfaction of Lyapunov conditions~(\ref{eq:Lyap}) (with $\mathbf{K}$ replaced by $\bar{\mathbf{K}}$) does not ensure the invariance of the closed-loop evolution of $(x,z)$ in the set~$\mathbf{X}$. Asymptotic stability is guaranteed only on the largest sub-level set contained in $\mathbf{X}$ of the function
\[
\bar{V}(x,z) = \sup_{\theta,\lambda}\{ V(x,z,\theta,\lambda) \mid (x,z,\theta, \lambda)\in \hat{\mathbf{K}}  \},
\]
where
\begin{align*}
\hat{\mathbf{K}} := \{(x,z, \theta, \lambda) \mid  \;  & \hat{g}(x,z, \theta, \lambda) \ge 0, \;  \hat{h}(x,z, \theta, \lambda) = 0, \\ & (x,z)\in\mathbf{X}\}.
\end{align*}

If we choose $V$ as a function of $(x,z)$ only, then trivially $\bar{V}(x,z) = V(x,z)$ and we can tractably find an inner approximation of this largest sub-level set. This can done by solving the following optimization problem:


 \begin{equation}\label{opt:maxLevelSet}
\begin{array}{rclll}
& \underset{\gamma\in\mathbb{R}_+, \{\sigma_{0,i}\},\{\sigma_{0,i}\}}{\mathrm{maximize}} & \gamma \\
& \hspace{-2.5cm} \mathrm{s.t.} &\hspace{-2.5cm} \psi_i(x,z) = \sigma_{0,i}(x,z)\! +\! \sigma_{1,i}(x,z)(\gamma-V(x,z)),\; \\
&& \hspace{-2.5cm}\{\sigma_{0,i}\},\: \{\sigma_{1,i}\} \;\;\;\; \text{SOS polynomials} \;\;\; \forall\,  i\in\{1,\ldots,n_\psi\}.
\end{array}
\end{equation}
Satisfaction of the first constraint implies that $\psi_i(x,z) \ge 0$ for all $(x,z)$ such that $V(x,z)\le \gamma$ and all $i\in\{1,\ldots,n_\psi\}$; therefore $\{(x,z)\mid V(x,z)\le\gamma\}\subset \bf X$ for any $\gamma$ feasible in~(\ref{opt:maxLevelSet}). Maximizing~$\gamma$ then maximizes the size of the inner approximation.

Problem~(\ref{opt:maxLevelSet}) is only quasi-convex because of the bilinearity between $\sigma_1$ and $\gamma$ but can be efficiently solved using a bi-section on $\gamma$. Indeed, for a fixed value of $\gamma$ problem~(\ref{opt:maxLevelSet}) is an SDP, typically of much smaller size than~(\ref{opt:sos}).

%

This immediately leads to the following theorem.
\begin{theorem}\label{thm:loca}
If a polynomial $V \in \Rb[x,z]$ satisfies~(\ref{eq:Lyap}) for all \[(x,z,\theta,\lambda,x^+,z^+,\theta^+,\lambda^+) \in \bar{\mathbf{K}}\] and $\gamma\in\mathbb{R}_+$ is feasible in~(\ref{opt:maxLevelSet}), then all trajectories of the closed-loop system~(\ref{eq:difIncl}) starting from the set $\{(x,z)\mid V(x,z)\le \gamma\}$ lie in the set $\mathbf{X}$ and $\lim_{t\to\infty} x_t = 0$.
\end{theorem} 
Since $\bar{\mathbf{K}}$ is defined by finitely many polynomial equalities and inequalities, the search for $V$ satisfying the conditions of Theorem~\ref{thm:loca} can be cast as an SOS problem completely analogous to~(\ref{opt:sos}).

\subsection{Performance analysis -- deterministic setting}\label{sec:perfDet}
In this section we analyze the performance of the controller~(\ref{eq:cont}) with respect to a given cost function. The performance is analyzed for all initial conditions belonging to a given set $\mathbf{X}$ defined in~(\ref{eq:Xdef}). In order to facilitate the performance analysis of the difference inclusion~(\ref{eq:difIncl}) we introduce a \emph{selection oracle}:
\begin{definition}[Selection oracle]\label{def:oracle} A selection oracle is any function $\oracle: 2^{\mathbb{R}^{n_u}\setminus \emptyset} \to \Rb^{n_u}$ satisfying $\oracle(A) \in A$ for all $A \subset \Rb^{n_u} \setminus \emptyset$.
\end{definition}
In words, a selection oracle is a function which selects one point from any nonempty subset of $\Rb^{n_u}$ (note that at least one such functions exists by the axiom of choice). The performance analysis of this section then pertains to the discrete-time recurrence
\begin{subequations}\label{eq:difIncl_oracle}
\begin{align}
x^+ & =  f_x(x,\oracle(\kappa(\mathbf{K}_s))), \\
z^+ &= f_z(z,f_y(x)),
\end{align}
\end{subequations}
and all results hold for an arbitrary selection oracle $\oracle$; hence in what follows we suppress the dependence of all quantities on the selection oracle. The cost function with respect to which we analyze performance is
\begin{equation}\label{eq:cost}
\mathcal{C}(x_0,z_0) = L\alpha^{\tau(x_0,z_0)} +  \sum_{t = 0}^{\tau(x_0,z_0)-1} \alpha^t l(x_t,u_t),
\end{equation}
$(x_t,z_t)_{t=0}^\infty $ is the solution to~(\ref{eq:difIncl}), $u_t = \oracle(\kappa(\mathbf{K}_{s_t}))$, $\alpha \in (0,1)$ is a discount factor, $l$ is a polynomial stage cost,
\begin{equation}\label{eq:tau}
\tau(x,z)  := \inf\{ t\in \{1,2,\ldots\} \mid ( x_t, z_t )  \notin \mathbf{X} ,\; (x_0,z_0) = (x,z)  \}
\end{equation}
is the first time that the state $( x_t, z_t ) $ leaves $\bf X$ (setting $\tau(x,z) = +\infty$ if $(x_t,u_t) \in \mathbf{X}$ for all $t$) and
\begin{equation}\label{eq:L}
L > \sup\{  l(x,u) \mid    (x,z)\in \mathbf{X}, u \in \kappa(\mathbf{K}_s)  \} / (1-\alpha)
\end{equation}
is a constant upper bounding the stage cost $l$ on $\bf X$ divided by $1-\alpha$. We assume that $L < \infty$, which is fulfilled if the projection of $\mathbf{X}$ on $\Rb^{n_x}$ is bounded and the set $\mathbf{K}_s$ is bounded for all $s = f_s(z,f_y(x))$ with $(x,z) \in \mathbf{X}$. A constant $L$ satisfying~(\ref{eq:L}) is usually easily found since $\mathbf{X}$ is known and the controller~(\ref{eq:cont}) is usually set up in such a way that it satisfies the input constraints of system~(\ref{eq:sys}), which are typically a bounded set of a simple form.

The reason for choosing~(\ref{eq:cost}) is twofold. First, $\mathcal{C}(x_0,z_0) = \sum_{t = 0}^{\infty} \alpha^t l(x_t,u_t) $ for all $(x_0,z_0)$ such that $(x_t,z_t) \in \mathbf{X}$ for all $t$; that is, whenever $(x_t,z_t) $ stays in the state constraint set $\mathbf{X}$ for all $t$, the cost~(\ref{eq:cost}) coincides with the standard infinite-horizon discounted cost. Second, $\mathcal{C}(x_0,z_0) \le L$ for all $(x_0,z_0) \in \bf X$; that is, the cost function is bounded on $\bf X$, which enables us to obtain polynomial upper and lower bounds on $\mathcal{C}$ (which is not possible if $\mathcal{C}$ is infinite outside the maximum positively invariant subset of~(\ref{eq:difIncl}) included in $\mathbf{X}$ as is the case for the standard infinite-horizon discounted cost).

In the rest of this section we derive polynomial upper and lower bounds on $\mathcal{C}(x,z)$. To this end define
\begin{align*}
\mathbf{\hat{K}}_c := \{(x,z, \theta, \lambda) \mid  \;  & \hat{g}(x,z, \theta, \lambda) \ge 0, \;  \hat{h}(x,z, \theta, \lambda) = 0, \\ & (x,z)\notin\mathbf{X}\}.
\end{align*}

The upper bound is based on the following lemma:
\begin{lemma}\label{lem:ub}
If 
\begin{align}\label{eq:ub_V}
&V(x,z,\theta,\lambda) - \alpha V(x^+,z^+,\theta^+,\lambda^+) - l(x,\kappa(\theta))  \ge 0 \\
&\nonumber \hspace{2cm} \forall\, (x,z,\theta,\lambda, x^+,z^+,\theta^+,\lambda^+) \in \bf \bar{K},
\end{align} 
\begin{equation}\label{eq:ub_V_L}
V(x,z,\theta,\lambda) \ge L\quad \forall\, (x,z,\theta,\lambda) \in \mathbf{\hat{K}}_c
\end{equation} 
and
\begin{equation}\label{eq:ub_barV}
\overline V(x,z) \ge V(x,z,\theta,\lambda) \quad \forall\,    (x,z,\theta,\lambda) \in \bf \hat{K}  ,
\end{equation}  then $\overline{V}(x,z) \ge \mathcal{C}(x,z)$ for all $(x,z) \in \bf X$.
\end{lemma}
\begin{proof}
See the Appendix.
\end{proof}

The lower bound is based on the following lemma:
\begin{lemma}\label{lem:lb}
If 
\begin{align}\label{eq:lb_V}
&V(x,z,\theta,\lambda) - \alpha V(x^+,z^+,\theta^+,\lambda^+) - l(x,\kappa(\theta))  \le 0 \\
&\nonumber \hspace{2cm} \forall\, (x,z,\theta,\lambda, x^+,z^+,\theta^+,\lambda^+) \in \bf \bar{K},
\end{align} 
\begin{equation}\label{eq:lb_V_L}
V(x,z,\theta,\lambda) \le L\quad \forall\, (x,z,\theta,\lambda) \in \mathbf{\hat{K}}_c
\end{equation} 
and
\begin{equation}\label{eq:lb_barV}
\underline V(x,z) \le V(x,z,\theta,\lambda) \quad \forall\,    (x,z,\theta,\lambda) \in \bf \hat{K}  ,
\end{equation}  then $\underline{V}(x,z) \le \mathcal{C}(x,z)$ for all $(x,z) \in \bf X$.
\end{lemma}
\begin{proof}
Analogous to the proof of Lemma~\ref{lem:ub}.
\end{proof}
%

The previous two lemmas lead immediately to optimization problems providing upper and lower bounds on $\mathcal{C}(x,z)$.

 An upper bound on $\mathcal{C}(x,z)$ is given by the following optimization problem:
 \begin{equation}\label{opt:C_ub}
\begin{array}{rclll}
& \underset{V,\overline{V}}{\mathrm{minimize}} & \displaystyle \int_\mathbf{X} \overline{V}(x,z) \rho(x,z) d(x,z) \vspace{1.2mm}\\
& \mathrm{s.t.} & (\ref{eq:ub_V}), (\ref{eq:ub_V_L}), (\ref{eq:ub_barV}),
\end{array}
\end{equation}
where $\rho(x,z)$ is a user-defined nonnegative weighting function allowing one to put a different weight on different initial conditions. Typical examples are $\rho(x,z) = 1$ or $\rho(x,z)$ equal to the indicator function of a certain subset of $\mathbf{X}$ (see Example~\ref{ex:bilinear}).

A lower bound on $\mathcal{C}(x,z)$ is given by the following optimization problem:
 \begin{equation}\label{opt:C_lb}
\begin{array}{rclll}
& \underset{V,\underline{V}}{\mathrm{maximize}} & \displaystyle \int_\mathbf{X} \underline{V}(x,z) \rho(x,z) d(x,z) \vspace{1.2mm}\\
& \mathrm{s.t.} & (\ref{eq:lb_V}), (\ref{eq:lb_V_L}), (\ref{eq:lb_barV}).
\end{array}
\end{equation}

In both optimization problems, the optimization is over continuous functions $(V,\overline{V})$ or $(V,\underline{V})$; in order to make the problems tractable we restrict the class of functions to polynomials of predefined degrees and replace the nonnegativity conditions (\ref{eq:ub_V}), (\ref{eq:ub_V_L}), (\ref{eq:ub_barV}) or (\ref{eq:lb_V}), (\ref{eq:lb_V_L}), (\ref{eq:lb_barV}) by sufficient SOS conditions (see Section~\ref{sec:SOSprog}). For (\ref{eq:ub_V}), (\ref{eq:ub_barV}) and (\ref{eq:lb_V}), (\ref{eq:lb_barV}), these conditions are completely analogous to~(\ref{eq:sos_eqs}). For (\ref{eq:ub_V_L}) and (\ref{eq:lb_V_L}) we have to deal with the condition $(x,z) \notin \bf X$. A sufficient condition for~(\ref{eq:ub_V_L}) is
\begin{align}\label{eq:sufCondbnd}
& \hspace{-0.21mm} V(x,z,\theta, \lambda) - L  = -\sigma_{\psi_i}(\zeta)\psi_i(x,z) + \sigma_0(\zeta)  \\   & + \sigma_1(\zeta)^\top \hat{g}(x,z,\theta,\lambda) + \bar{p}_1(\zeta)^\top \hat{h}(x,z,\theta,\lambda)\; \forall\, i \in \{1,\ldots, n_\psi\}, \nonumber 
\end{align}
 where
 \[
 \zeta:= (x,z,\theta, \lambda),
 \]
$\sigma_0$, $\sigma_1$ and $\sigma_{\psi_i}$'s are SOS and $\bar{p}_1$ is a polynomial. For each $i\in \{1,\ldots,n_\psi\}$ this condition implies that $ V(x,z,\theta, \lambda) - L  \ge 0 $ on
\begin{align*}
\mathbf{K}_{c,i} = \{(x,z, \theta, \lambda) \mid  \;  & \hat{g}(x,z, \theta, \lambda) \ge 0, \;  \hat{h}(x,z, \theta, \lambda) = 0, \\ & \psi_i(x,z) \le 0 \}.
\end{align*}
Since $\cup_{i=1}^{n_\psi} \mathbf{K}_{c,i} = \mathbf{K}_c$, the condition (\ref{eq:sufCondbnd}) indeed implies~(\ref{eq:ub_V_L}). A sufficient condition for (\ref{eq:lb_V_L}) is obtained by replacing the left-hand side of~(\ref{eq:sufCondbnd}) by $L - V(x,z,\theta, \lambda) $.

Therefore, all constraints of the optimization problems~(\ref{opt:C_ub}) and (\ref{opt:C_lb}) can be enforced through sufficient SOS conditions. The objective function is linear in the coefficients of the polynomials $\overline{V}$ or $\underline{V}$ and can be evaluated in closed form for simple sets $\mathbf{X}$; see Section~\ref{sec:SOSprog} for details.

In conclusion, by restricting the class of decision variables in (\ref{opt:C_ub}) and (\ref{opt:C_lb}) to polynomials of a prescribed degree and replacing the nonnegativity constraints by sufficient SOS conditions, the problems (\ref{opt:C_ub}) and (\ref{opt:C_lb}) become SOS problems with linear objective and hence immediately translate to SDPs.

\subsection{Performance analysis -- stochastic setting}\label{sec:perfStoch}
A small modification of the developments from the previous section allows us to analyze the performance in a stochastic setting, where~(\ref{eq:difIncl}) is replaced by
\begin{subequations}\label{eq:difIncl_stoch}
\begin{align}
x^+ &  = f_x(x,\oracle(\kappa(\mathbf{K}_s)),w), \\
z^+ &= f_z(z,f_y(x,v)),
\end{align}
\end{subequations}
where $\oracle$ is an arbitrary selection oracle (see Definition~\ref{def:oracle}) and $(w,v)$ is an iid (with respect to time) process and measurement noise with known joint probability distributions $P_{w,v}$, i.e., \[ \mathbb{P}(w \in A, v \in B) = P_{w,v}(A\times B)\] for all Borel sets $A\subset \Rb^{n_w}$ and $B \subset \Rb^{n_v}$. We analyze the performance with respect to the cost function
\begin{equation}\label{eq:costStoch}
\mathcal{C}_\mr{s}(x_0,z_0) = \E \Big\{ L\alpha^{\tau(x_0,z_0)}\} + \hspace{-2mm}  \sum_{t = 0}^{\tau(x_0,z_0)-1} \alpha^t l(x_t,u_t) \Big \},
\end{equation}
where $\tau(x_0,z_0)$ defined in~(\ref{eq:tau}) is now a random variable and $L$ satisfies~(\ref{eq:L}). The expectation in~(\ref{eq:costStoch}) is over the realizations of the stochastic process $(w_t,v_t)_{t=0}^\infty$. The rationale behind~(\ref{eq:costStoch}) is the same as behind~(\ref{eq:cost}).

The stochastic counterpart to Lemma~\ref{lem:ub} reads
\begin{lemma}\label{lem:ub_stoch}
If 
\begin{align}\label{eq:ub_V_stoch}
&\overline{V}(x,z) \!-\! \alpha \E \overline{V}(f_x(x,\kappa(\theta),w),f_z(z,f_y(x,v))) \! - \! l(x,\kappa(\theta))  \nonumber \\ & \hspace{3.5cm} \ge 0  \quad \forall\, (x,z,\theta,\lambda) \in \bf \hat{K},
\end{align} 
and
\begin{equation}\label{eq:ub_V_L_stoch}
\overline{V}(x,z) \ge L\quad \forall\, (x,z) \in \mathbf{X}^c,
\end{equation} 
then $\overline{V}(x,z) \ge \mathcal{C}_\mr{s}(x,z) $ for all $(x,z) \in \bf X$,
\end{lemma}
\noindent where $\mathbf{X}^c$ is the complement of $\mathbf{X}$.
\begin{proof}
See the Appendix.
\end{proof}

 The stochastic counterpart to Lemma~\ref{lem:lb} reads
\begin{lemma}\label{lem:lb_stoch}
If 
\begin{align}\label{eq:lb_V_stoch}
&\underline{V}(x,z) - \alpha \E \underline{V}(f_x(x,\kappa(\theta),w),f_z(z,f_y(x,v))) - l(x,\kappa(\theta))  \nonumber \\ & \hspace{3.5cm} \le 0  \quad \forall\, (x,z,\theta,\lambda) \in \bf \hat{K},
\end{align} 
and
\begin{equation}\label{eq:lb_V_L_stoch}
\underline{V}(x,z) \le L\quad \forall\, (x,z) \in \mathbf{X}^c,
\end{equation} 
then $\underline{V}(x,z) \le \mathcal{C}_\mr{s}(x,z) $ for all $(x,z) \in \bf X$.
\end{lemma}
\begin{proof}
Similar to the proof of Lemma~\ref{lem:ub_stoch}.
\end{proof}
\noindent Upper and lower bounds on $\mathcal{C}_\mr{s}$ are then obtained by
 \begin{equation}\label{opt:C_ub_stoch}
\begin{array}{rclll}
& \underset{\overline{V}}{\mathrm{minimize}} & \displaystyle \int_\mathbf{X} \overline{V}(x,z) \rho(x,z)d(x,z) \vspace{1.2mm}\\
& \mathrm{s.t.} & (\ref{eq:ub_V_stoch}), (\ref{eq:ub_V_L_stoch}),
\end{array}
\end{equation}
and
 \begin{equation}\label{opt:C_lb_stoch}
\begin{array}{rclll}
& \underset{\underline{V}}{\mathrm{maximize}} & \displaystyle \int_\mathbf{X} \underline{V}(x,z) \rho(x,z)d(x,z) \vspace{1.2mm}\\
& \mathrm{s.t.} & (\ref{eq:lb_V_stoch}), (\ref{eq:lb_V_L_stoch}),
\end{array}
\end{equation}
where $\rho(x,z)$ is a given nonnegative weighting function. Polynomial upper and lower bounds on $ \mathcal{C}_\mr{s}$ are obtained by restricting the functions $\overline{V}$ and $\underline{V}$ to polynomials and replacing the nonnegativity constraints by sufficient SOS constraints in exactly the same fashion as in the deterministic setting. The expectation in the constraints~(\ref{eq:ub_V_stoch}) and (\ref{eq:lb_V_stoch}) is handled as follows: Given a polynomial $p(x,w,v) = \sum_{\alpha, \beta,\gamma} p_{(\alpha,\beta,\gamma)}x^\alpha w^\beta v^\gamma$ with coefficients $\{p_{(\alpha,\beta,\gamma)}\}$ indexed by multiindices $(\alpha,\beta,\gamma)$, we have
\begin{align*}
&\E\, p(x,w,v) = \int p(x,w,v) dP_{w,v}(w,v) \\ &= \sum_{\alpha, \beta,\gamma} p_{(\alpha,\beta,\gamma)} x^\alpha \int w^\beta v^\gamma d P_{w,v}(w,v),
\end{align*}
where the moments  $\int w^\beta v^\gamma d P_{w,v}(w,v)$ are fixed numbers and can be precomputed offline. Hence, the expectation in~(\ref{eq:ub_V_stoch}) and (\ref{eq:lb_V_stoch}) is linear in the decision variables, as required, and is available in closed form provided that the moments of $P_{w,v}$ are known.

\begin{remark}
Note that in problems~(\ref{opt:C_ub_stoch}) and (\ref{opt:C_lb_stoch}) we use only one function $\overline{V}$ and $\underline{V}$ instead of pairs of functions $(V,\overline{V})$ and $(V,\underline{V})$ in problems~(\ref{opt:C_ub}) and (\ref{opt:C_lb}). Using a pair of functions gives more degrees of freedom and hence smaller conservatism (see Remark~\ref{rem:other_vars}) of the upper and lower bounds, but is difficult to use in the stochastic setting because of the need to evaluate the expectation of a function of $(\theta^+,\lambda^+)$ which has an unknown dependence on $(w,v)$. In order to overcome this, one would either have  to impose additional assumptions or resort to a worst-case approach.
\end{remark}

\subsection{Robustness analysis -- global $\mathcal{L}_2$ gain, ISS}\label{sec:perfRob}
In this section we describe how to analyze performance in a robust setting in terms of the $\mathcal{L}_2$ gains from $w$ and $v$ to a performance output
\begin{equation}\label{eq:perfOut}
\hat{y} = f_{\hat{y}}(x),
\end{equation}
where $f_{\hat{y}}$ is a polynomial. We assume the same dynamics~(\ref{eq:difIncl_stoch}) as in Section~(\ref{sec:perfStoch}) but now all that is known about $w$ and $v$ is that they take values in a given (possibly state-dependent) set
\begin{equation}\label{eq:W}
\mathbf{W}(x,z) = \{(w,v)\in\mathbb{R}^{n_w}\times \mathbb{R}^{n_v} \mid \psi_w(x,z,w,v) \ge 0\},
\end{equation}
where each component of $\psi_w:\mathbb{R}^{n_x+n_z+n_w+n_v} \to \Rb^{n_{\psi_w}}$ is a polynomial in $(x,z,w,v)$. Note that we do not a priori assume that the set $\mathbf{W}(x,z)$ is compact.

Defining
\begin{align}\label{eq:Kw}
 \mathbf{K}_w \!=\!\big\{ (x,z, \theta, \lambda, w, v,x^+ z^+, \theta^+, \lambda^+,w^+,v^+)  \mid  \\ 
 & \hspace{-6.0cm} \nonumber \hat{h}(x,z,\theta,\lambda) = 0, \; \hat{g}(x,z, \theta, \lambda) \ge 0, \\
& \hspace{-6.0cm}  \nonumber  \;  \hat{h}(x^+,z^+, \theta^+, \lambda^+) = 0,\;  \hat{g}(x^+,z^+, \theta^+, \lambda^+) \ge 0, \\ &\hspace{-6cm} \psi_w(x,z,w,v)\ge 0,\; \psi_w(x^+,z^+,w^+,v^+)\ge 0, \nonumber \\ \nonumber & \hspace{-6cm} x^+ - f_x(x,\kappa(\theta),w) = 0,\;z^+ - f_z(z,f_y(x,v)) = 0
 \big\},
\end{align} 
and
\begin{align}\label{eq:Kw_hat}
\nonumber \hat{\mathbf{K}}_w \!=\!\big\{ (x,z, \theta, \lambda, w, v)  \mid \;\,  &\hat{h}(x,z,\theta,\lambda) = 0,\; \hat{g}(x,z, \theta, \lambda) \ge 0,  \\ &\psi_w(x,z,w,v)\ge 0 \},
\end{align} 
we can seek a function $V$ such that
\begin{align}\label{eq:LyapDec_w}
\nonumber  &V(x^+, z^+,\theta^+, \lambda^+,w^+,v^+) - V(x, z,\theta, \lambda,w,v) \le \\ 
& \hspace{3.6cm} -\| f_{\hat{y}}(x) \| ^2_2 + \alpha_w\|w\|_2^2 + \alpha_v\|v\|_2^2 \nonumber\\
&\hspace{0.9cm} \forall\;(x,z, \theta, \lambda, w, v,x^+ z^+, \theta^+, \lambda^+,w^+,v^+)  \in \mathbf{K}_w,
\end{align}
\begin{align}\label{eq:LyapNonNeg_w}
V(x,z,\theta, \lambda,w,v) \ge 0\;\;\;\; \forall\; (x,z,\theta, \lambda,w,v) \in \hat{\mathbf{K}}_w.
\end{align}
and
\begin{align}\label{eq:Lyapzer_w}
V(0,0,\theta, \lambda,w,v) = 0\;\;\;\; \forall\; (0,0,\theta, \lambda,w,v) \in \hat{\mathbf{K}}_w.
\end{align}

The following lemma and its immediate corollary links the satisfaction of (\ref{eq:LyapDec_w}),~(\ref{eq:LyapNonNeg_w}) and~(\ref{eq:Lyapzer_w}) to the $\mathcal{L}_2$ gain from $w$ and $v$ to $\hat{y}$.
\begin{lemma}\label{lem:rob}
If $V$ satisfies~(\ref{eq:LyapDec_w}),~(\ref{eq:LyapNonNeg_w}) for some $\alpha_w \ge 0$ and $\alpha_v \ge 0$, then
\begin{align}\label{eq:L2bounded}
\sum_{t=0}^\infty \| \hat{y}_t\|_2^2 & \le V(x_0,z_0,\theta_0,\lambda_0,w_0,v_0) \\ & + \alpha_w  \sum_{t=0}^\infty \| w_t\|_2^2 +  \alpha_v \sum_{t=0}^\infty \| v_t\|_2^2. \nonumber
\end{align}
\end{lemma}
\begin{proof}
See the Appendix.
\end{proof}

\begin{corollary}\label{lem:L2gain}
If $V$ satisfies~(\ref{eq:LyapDec_w}),~(\ref{eq:LyapNonNeg_w}) and~(\ref{eq:Lyapzer_w}) for some $\alpha_w \ge 0$ and $\alpha_v \ge 0$, then the $\mathcal{L}_2$ gain from $w$ to $\hat{y}$ respectively from $v$ to $\hat{y}$ is bounded by $\alpha_w$ respectively $\alpha_v$.
\end{corollary}
\begin{proof}
Follows by setting $(x_0,z_0) = (0,0)$ and using~(\ref{eq:Lyapzer_w}) which implies that \[V(x_0,z_0,\theta_0,\lambda_0,w_0,v_0) = 0\] in~(\ref{eq:L2bounded}).
\end{proof}

Minimization of an upper bound on the $\mathcal{L}_2$ gain from $w$ and $v$ to $\hat{y}$ is then achieved by the following optimization problem:
 \begin{equation}\label{opt:robust}
\begin{array}{rclll}
& \underset{V,\alpha_w,\alpha_v}{\mathrm{minimize}} & \alpha_w + \gamma\alpha_v \vspace{1.2mm}\\
& \mathrm{s.t.} & (\ref{eq:LyapDec_w}), (\ref{eq:LyapNonNeg_w}), (\ref{eq:Lyapzer_w}),
\end{array}
\end{equation}
where the parameter $\gamma \ge 0$ trades off the minimization of the $\mathcal{L}_2$ gains from $w$ to $\hat{y}$ and from $v$ to $\hat{y}$.

\begin{remark}\label{rem:ISS}
If instead of~(\ref{eq:LyapNonNeg_w}) we require $ V(x,z,\theta, \lambda,w,v) \ge \| x\|^2$ for all $ (x,z,\theta, \lambda,w,v) \in \hat{\mathbf{K}}_w$, then this along with~(\ref{eq:LyapDec_w}) implies that the system~(\ref{eq:difIncl_stoch}) is input-to-state stable (ISS) with respect to the input $(w,v) \in \mathbf{W}(x,z)$.
\end{remark}

Since the sets $\mathbf{K}_w$ and $\hat{\mathbf{K}}_w$ are defined by finitely many polynomial equalities and inequalities, we can find upper bounds on the $\mathcal{L}_2$ gains $\alpha_w$ and $\alpha_v$ by restricting $V$ to be a polynomial of a prescribed degree and replacing the nonnegaivity constraints of~(\ref{opt:robust}) by sufficient SOS conditions according to Section~\ref{sec:SOSprog}.


\section{Computational aspects and practical guidelines}\label{sec:compAspects}

As discussed in Section~\ref{sec:SOSprog}, the constraint~(\ref{eq:sosGen}) translates to an SDP constraint. There is a natural tradeoff between computational complexity and the richness of the class of functions satisfying~(\ref{eq:sosGen}). This tradeoff is, for the most part, controlled by the size of the polynomial basis parametrizing the function $V$ and by the vector of polynomials $\beta(x)$ in~(\ref{eq:sos_eq}) associated to the SOS multipliers. In general, the larger $\beta(x)$ (in terms of the number of components), the larger the set of functions covered, but the higher the computational complexity. A typical choice for $\beta(x)$ is the vector of all multivariate monomials up to a given total degree. Another choice may be the set of all multivariate monomials up  to a given total degree corresponding to a given subset of the variables $x = (x_1,\ldots, x_n)$; this is used in the numerical examples of this paper.

More specific monomial selection techniques require some insight into the problem structure. Fortunately, there exist automatic monomial reduction techniques~(e.g., the Newton polytope~\cite{sturmfels}) which discard those monomials in $\beta(x)$ which cannot appear in the decomposition of $\sigma(x)$. It is also possible to directly reduce the size of the SDP arising from~(\ref{eq:sos_eq}) using the facial reduction algorithm of~\cite{facialReduction}; the benefit of this approach is that it works at the SDP level and is therefore not limited to the situation where $\beta(x)$ is a subset of the monomial basis, allowing one to use numerically better conditioned bases (e.g., the Chebyshev basis).

The transformation from the abstract form~(\ref{eq:sos_eq}) can be carried out automatically using freely available modeling tools such as Yalmip~\cite{yalmip}, SOSTOOLS~\cite{sostools} or SOSOPT~\cite{sosopt} and solved using SDP solvers such as SeDuMi~\cite{sedumi} or MOSEK~\cite{mosek}.

It should be noted that at present the approach is tractable for problems of moderate size only unless special effort is made in terms of exploiting the problem structure, e.g., by using a problem-specific monomial selection heuristic (exploiting, e.g., the sparsity or symmetries of the problem) or by using a custom SDP solver. In addition, the scalability of the approach could be improved if alternative sufficient nonnegativity conditions were used instead of the standard (Putinar-type) SOS conditions. For example, the recently proposed DSOS and SDSOS cones of nonnegative polynomials~[1] or the bounded-degree hierarchy of~[21] have shown promising results in other application areas. In addition, a more efficient implementation of the monomial reduction itself and polynomial handling in general would also allow the approach to scale higher as these may account for a significant portion of the computation time.

When implementing the method in practice it is advisable to start with a high level Yalmip or SOSOPT implementation and a simple choice of a monomial basis, e.g., monomials in all state variables for the Lyapunov function and monomials in all variables involved for the SOS and polynomial multipliers. The degrees should be selected as low as possible, e.g., quadratic for the Lyapunov function and SOS multipliers and linear for polynomial multipliers. If this selection does not lead to a satisfactory certificate, the degrees should be increased and/or decision variables or Lagrange multipliers included in the Lyapunov function. If the computation time becomes an issue, a more sophisticated monomial selection and/or customization is required, as described above.

\section{Numerical examples}\label{sec:numEx}


This section illustrates the approach on two numerical examples. For the first two numerical examples we used Yalmip~\cite{yalmip} as the modeling tool; for the third example we used SOSOPT~\cite{sosopt}. The SDP solver used was MOSEK~\cite{mosek} for all three examples.

\subsection{Bilinear system + PI with saturation -- performance analysis}\label{ex:bilinear}
First we demonstrate the approach on a bilinear dynamical system
\begin{align*}
x^+ &= f_x(x,u) := \begin{bmatrix} 0.9x_1 + u + 0.2ux_1 \\ 0.85x_2 + x_1 \end{bmatrix}\\
y & =  f_y(x) := x_2
\end{align*}
controlled by a PI controller with input saturation given by
\begin{align*}
z^+ &= f_z(z,y) := z - k_{i}y\\
s & =  f_s(z,y) := k_{p} (z - y)
\end{align*}
with $k_p = 0.05$, $k_i = 0.02$. The control input is given by saturating $u$ on the input constraint set $U = [-0.5,0.5]$, i.e, $u = \mr{proj}_U(s) $. In addition the system is subject to the state constraints $\| x \| _\infty \le 10$. In view of~Section~\ref{sec:inputSat}, this set up can be analyzed using the presented method. The goal is to estimate the performance of this closed-loop system with respect to the cost function~(\ref{eq:cost}) with $l(x,u) = \|x \|^2 + u^2$, $\alpha  = 0.95$ and $L = (2\cdot 10^2 + 0.5^2) / (1 - \alpha) =  4.05\cdot 10^3$ chosen according to~(\ref{eq:L}). We estimate the performance using the optimization problem~(\ref{opt:C_ub}), where we consider $V$ as function of $(x,z)$ only and therefore do not need the upper bounding function $\overline{V}$. Assume that we are interested only in closed-loop performance for initial conditions starting from $\mathbf{X}' = \{ x \mid \| x \| _\infty \le 1 \} $ and $z = 0$ (i.e., zero integral component at the beginning of the closed-loop evolution), which is a strict subset of the set $\mathbf{X} := \{ (x,z) \mid \| x \| _\infty \le 10, z\in \Rb \} $. To this effect we minimize  $\int_{\mathbf{X} '} V(x,0)\, dx$ as the objective of (\ref{opt:C_ub}), which corresponds to setting $\rho(x,z) = \mathbb{I}_{\mathbf{X}'}(x)\delta_0(z)$, where $\mathbb{I}_{\mathbf{X}'}$ is the indicator function of $\mathbf{X}'$ and $\delta_0$ the dirac distribution centered at zero. We compare the upper bound obtained by solving~(\ref{opt:C_ub}) with the exact cost function evaluated on a dense grid of initial conditions in~${\mathbf{X} '} $ by forward simulation of the closed-loop system. The comparison is in Figure~(\ref{fig:perform1}); we see a relatively good fit over the whole region of interest ${\mathbf{X} '} $. The constraints~(\ref{eq:ub_V}) and (\ref{eq:ub_V_L}) of~(\ref{opt:C_ub}) were replaced with sufficient SOS conditions with SOS multipliers of degree four containing only monomials in $(x,z)$ and polynomial multipliers of degree three containing monomials in $(x,z,\theta,\lambda)$.

\begin{figure}[t!]
	\begin{picture}(40,40)
	
	\put(52,2){\includegraphics[width=50mm]{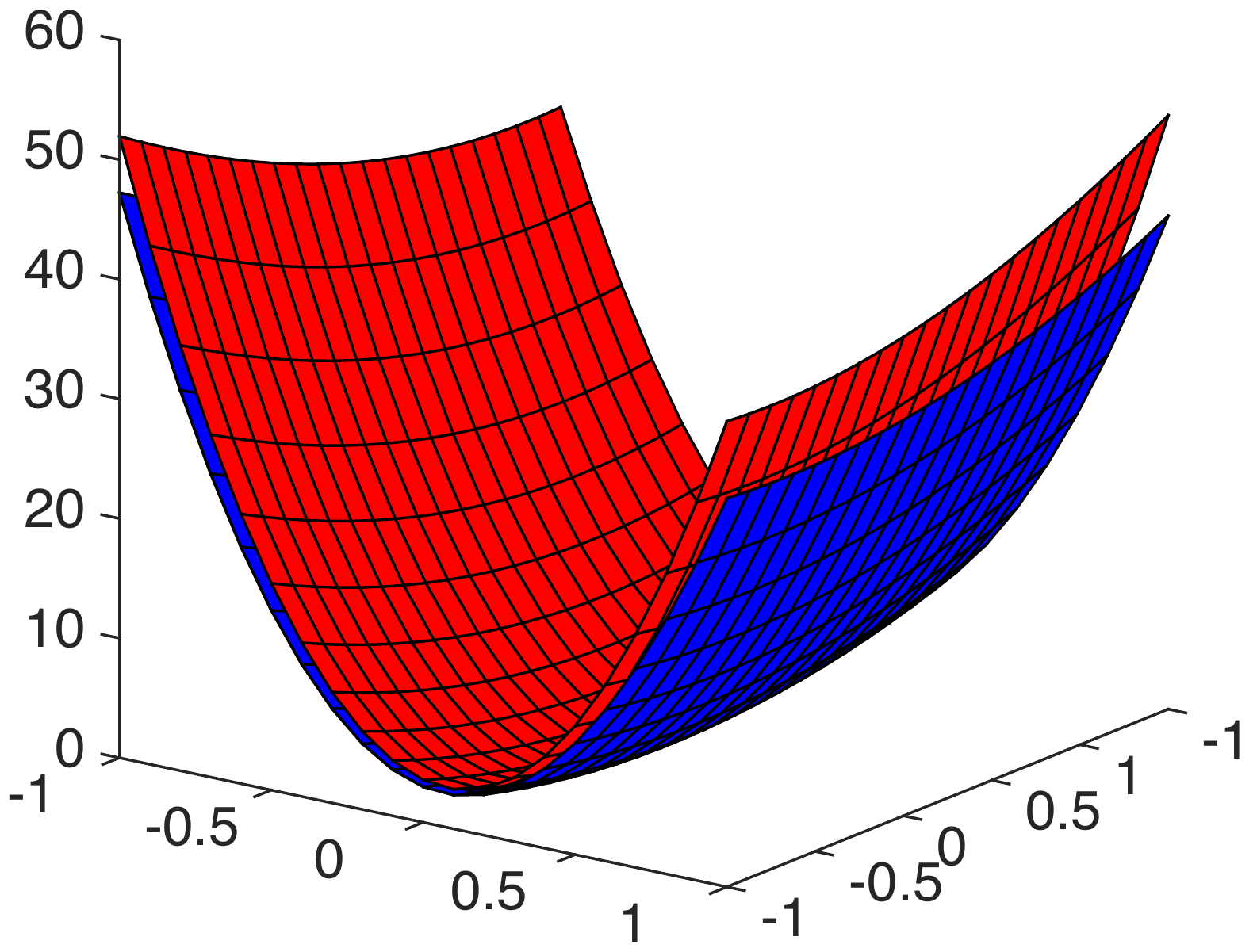}}

	\put(63,2.0){ $x_1$}
	\put(92.5,3.5){ $x_2$}

	\end{picture}
	\caption{\small Bilinear system performance bound -- Red: upper bound $V(x,0)$ of degree 6. Blue: true closed-loop cost $J(x,0)$.}
	\label{fig:perform1}
\end{figure}

\subsection{MPC + observer -- robust stability}
Consider the Quanser active suspension model in continuous-time 
\begin{align*}
\dot{x} &= A_c x + B_c u \\ 
y &= Cx
\end{align*}
 with
\[
\footnotesize A_c =  \begin{bmatrix} 0 & 1 & 0 & -1 \\  -K_s/M_s & -B_s/M_s & 0 & B_s/M_s \\ 0 & 0 & 0 & 1 \\  K_s/M_{us} & B_s/M_{us} & -K_{us}/M_{us} & -(B_s+B_{us})/M_{us} \end{bmatrix},
\]
\[
\footnotesize B_c = \begin{bmatrix} 0 & 1/M_s & 0 & -1/M_{us} \end{bmatrix}^T, \quad C = \begin{bmatrix}1 &0 &0& 0\\0& 0& 1& 0 \end{bmatrix},
\]
where $K_s = 1205$, $K_{us} = 2737$, $M_{us} = 1.5$, $B_s = 20$, $B_{us} = 20$ and the mass $M_s$ is unknown and possibly time-varying in the interval $[2.85,4]$. After discretization\footnote{The matrices $A_0$, $A_1$, $B_0$, $B_1$ were found as a least-squares fit of the continuous-time dynamics discretized on a grid of values of $w\in [1/4,1/2.85]$.} with sampling period~$0.01$, this model can be written as 
\begin{align*}
x^+  & =  (A_0 + A_1w)x + (B_0+B_1w)x =: f_x(x,u,w) \\ 
y  &= Cx =: f_y(x),
\end{align*}
 where $w := 1/M_s \in [1/4,1/2.85]$.

The discretized system is controlled by an MPC controller which minimizes along a prediction horizon $N$ the cost function $x_N^T P x_N+\sum_{i=0}^{N-1} x_i^T Q x_i + u_i^T R u_i$ with $Q = I$ and $R = 20$ subject to the input constraints $|u| \le u_{\mathrm{max}} =  250$ and nominal dynamics \begin{equation}\label{eq:nomDynamics}
x^+ = \bar{A}x + \bar{B}u,
\end{equation}
where $\bar A = A_0 + A_1\bar{w}$ and $\bar B = B_0 + B_1\bar{w}$ with $\bar{w} = 0.3004$ being the mid point of the range of values of the uncertain parameter $w$. The matrix $P$ is the unique positive definite solution to the discrete algebraic Riccati equation associated to $(\bar A,\bar B, Q, R)$.

The MPC controller takes as its input (i.e., as the initial state of the recurrence~(\ref{eq:nomDynamics})) the estimate of the state $x$, denoted by $z$, provided by a full order Luenberger observer with the dynamics
\[
z^+ =  \bar{A}z + \bar{B}u + K_{\mathrm{est}}(Cz - y) =: f_z(z,y) ,
\]
with\footnote{The gain $K_{\mathrm{est}}$ was obtained as the optimal Kalman filter gain with measurement and noise covariance matrices equal to the identity matrix.}
\[K_{\mathrm{est}} =
\begin{bmatrix}
 0.6137 &   -0.1594 &    0.7947 &    0.0585
\end{bmatrix}.
\]
The output mapping $s = f_s(z,y)$ from Figure~\ref{fig:probSetup} is in this case given by $f_s(z,y) = z$.

 This problem is expressed in a dense form (i.e., the state is eliminated using the nominal dynamics equation); hence at each time step of the closed-loop operation the MPC controller solves the optimization problem
 \begin{equation}
\label{opt:uncertainConcrete}
\begin{array}{ll}
\underset{\theta\in \Rb^{n_\theta}}{\mbox{minimize}} & \theta ^T H \theta + s^T F \theta  \\
\mbox{subject to} & \theta \le u_{\mathrm{max}}\\ 
&							\hspace{-3mm}-\theta \le u_{\mathrm{max}}	,
\end{array}
\end{equation}
which is parametrized by the estimated state $s = z$ and $H$ and $F$ are appropriate matrices that are readily obtained from $\bar A$ and $\bar B$. The decision variable $\theta \in \mathbb{R}^N$ is the predicted sequence of control inputs along the prediction horizon $N$. Problem~(\ref{opt:uncertainConcrete}) is of the form of Problem~(\ref{opt:gen}) and hence we can use the results of Section~\ref{sec:perfRob} to seek an ISS Lyapunov function $V$ quadratic in the variables $(x,z)$ (see Remark~\ref{rem:ISS}) while minimizing the $\mathcal{L}_2$ gain $\alpha_w$ using the optimization problem~(\ref{opt:robust}) (with $\alpha_v = 0$). The problem~(\ref{opt:robust}) is feasible (for all values of $N$ tested) when we take the SOS multipliers $\sigma_1$, $\sigma_2$ in equation~(\ref{eq:sos_decrease}) of degree two in $(x,z,\theta,\lambda)$ and the polynomial multipliers $p_1$, $p_2$ of degree one in $(x,z,\theta,\lambda)$. In Eq.~(\ref{eq:sos_nonneg}) we set all multipliers to zero except for $\bar{\sigma}_0$. The optimal $\mathcal{L}_2$ gain $\alpha_w$ is equal to zero (up to numerical errors) for all values of $N$ tested, showing closed-loop global robust asymptotic stability (i.e., convergence $\|x_k\| \to 0$ for any sequence $\{w_k\in [1/4,1/2.85]\}_{k=0}^\infty$ and for any initial estimate of the state). Figure~\ref{fig:uncertain} shows a sample trajectory of $\|x_k\|$, $\|z_k\|$, $V(x_k,z_k)$, $u_k$ and $w_k$ for $N = 5$. Note that, for numerical reasons, the control input was scaled to $[-1,1]$ before solving the verification problem and as such is reported in Figure~\ref{fig:uncertain}. Computation time breakdown for different values of $N$ is reported in Table~\ref{tab:time1}.

We also note that various modifications of the setup can be readily tested. For example, verifying robust stability when we use $[x(1),z(2),x(3),z(4)]^T$ as the initial state for the predictions of the MPC controller instead of $z$ (i.e., we use the two available state measurements instead of their estimates) amounts to changing one line of the Yalmip code and produces very similar results.

\begin{figure*}[th]
\begin{picture}(140,40)
\put(-10,2){\includegraphics[width=40mm]{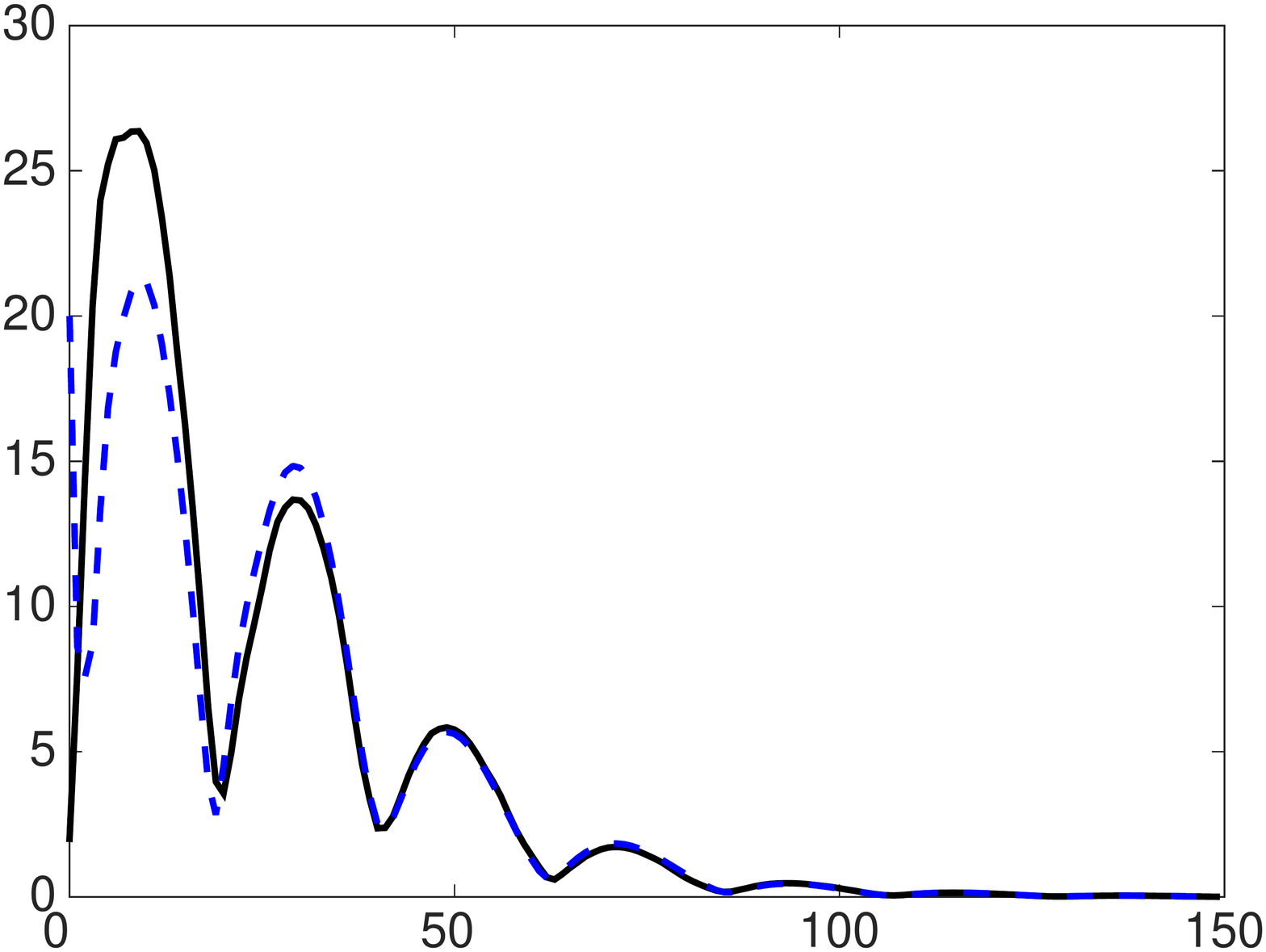}} 
\put(32,2){\includegraphics[width=40mm]{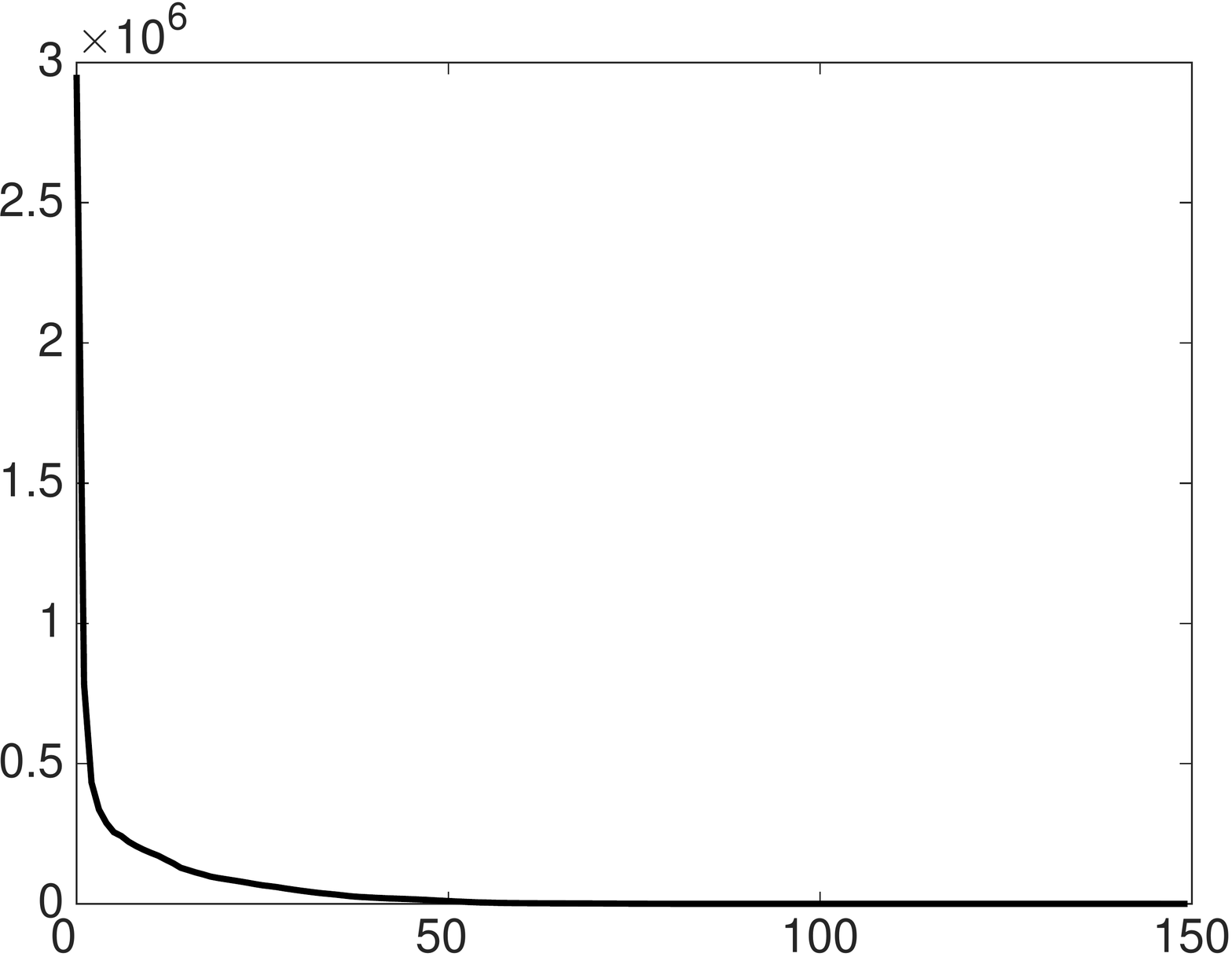}}
\put(77,2){\includegraphics[width=40mm]{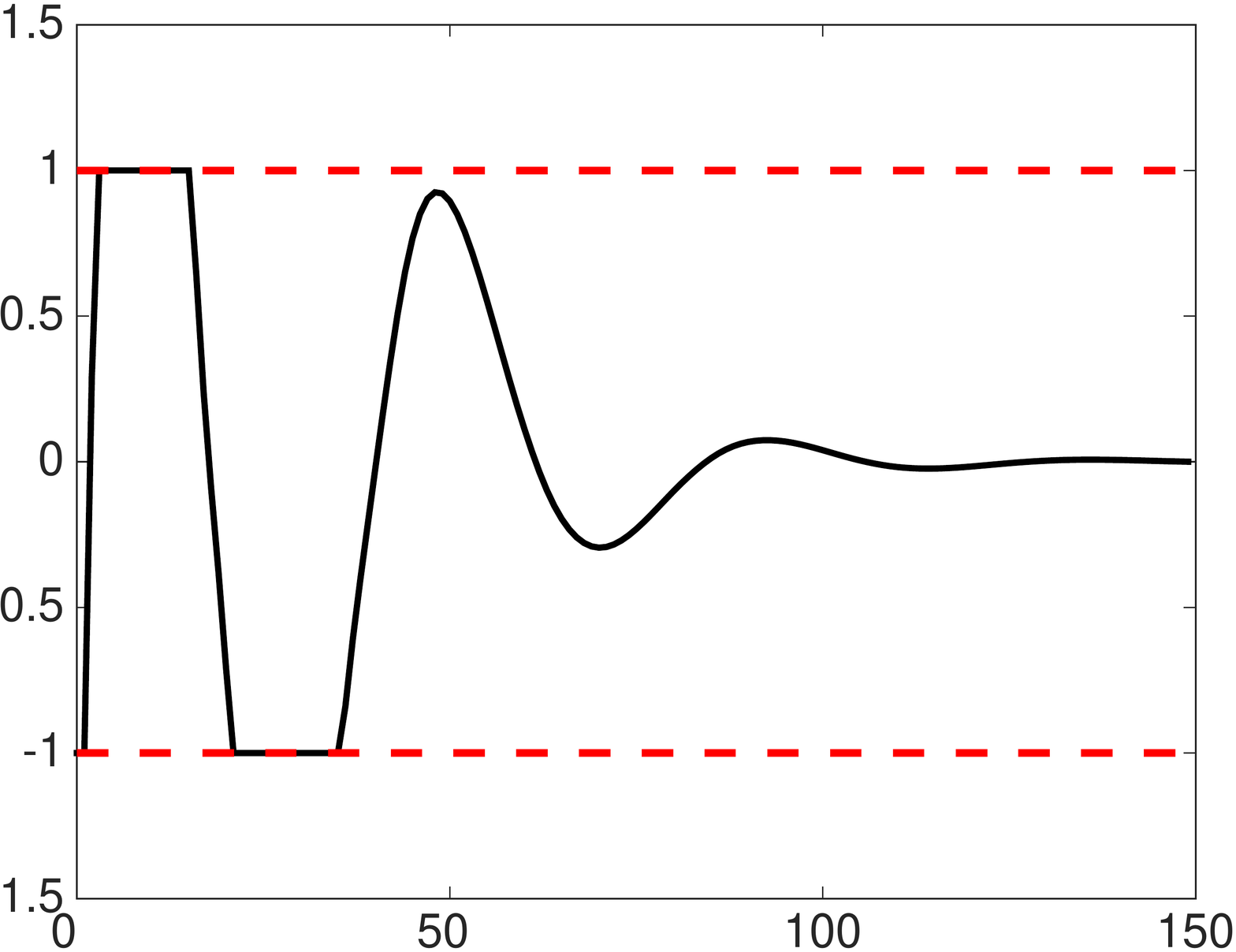}}
\put(122,2){\includegraphics[width=40mm]{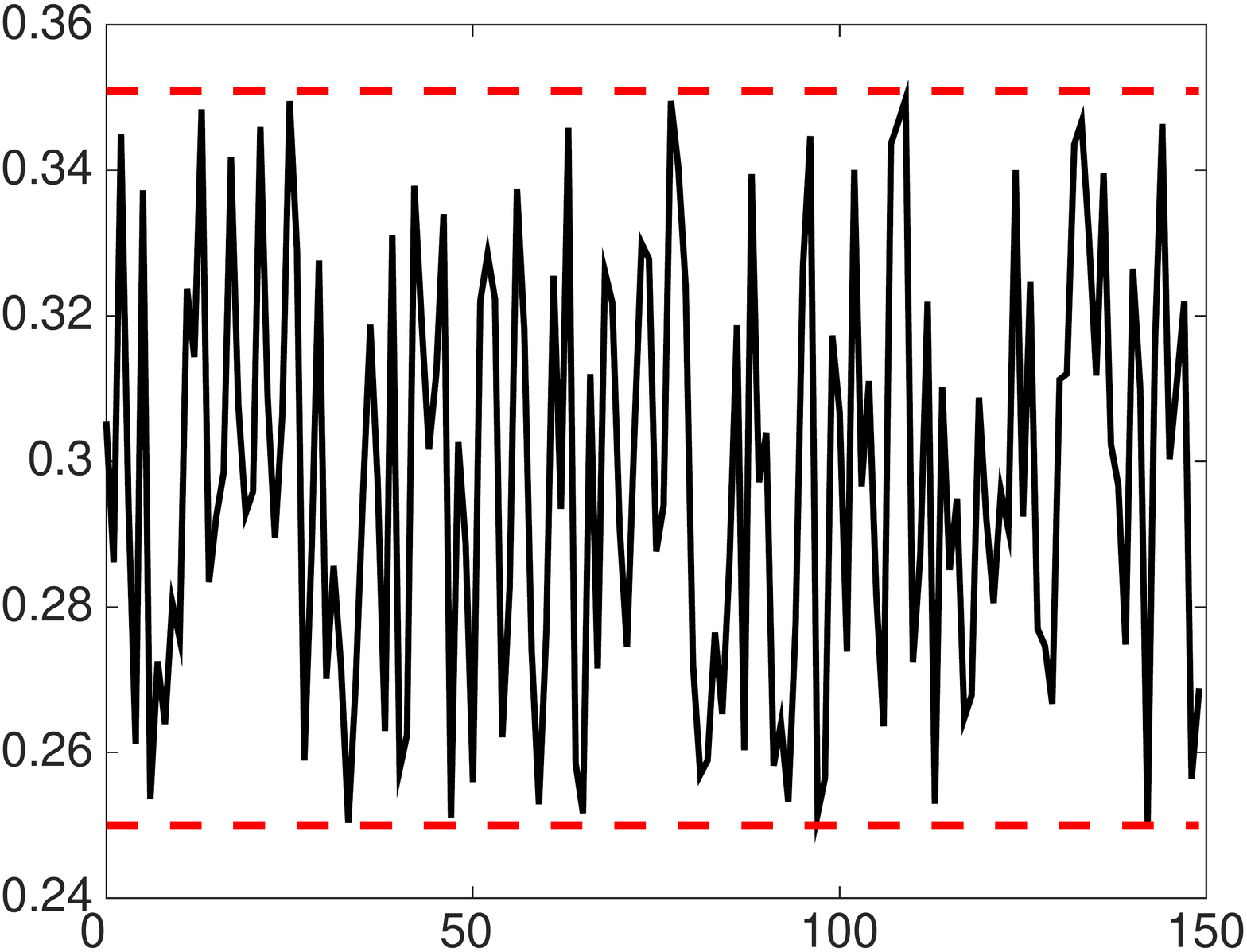}}

\put(10,0.5){\footnotesize $k$}
\put(52.5,0.5){\footnotesize $k$}
\put(97,0.5){\footnotesize $k$}
\put(142,0.5){\footnotesize $k$}

\put(2,20){ $\|x_k\|_2,\; \|z_k\|_2$}
\put(42,20){ $V(x_k)$}
\put(94,20){ $u_k$}
\put(133,26.5){ $w_k$}

\end{picture}
\caption{Global asymptotic stability of an uncertain system -- a trajectory starting from the initial condition $x_0 = [1\; 1\; 1\; 1]^T$, $z_0 = [-10\;-10\;-10\;-10]^T$ of the norm of the state $\|x_k\|_2$ (black) and the norm of the state estimate $\| z_k\|_2 $ (blue), the Lyapunov function $V(x_k,z_k)$, the control input $u_k$ and the disturbance $w_k$.}
\label{fig:uncertain}
\end{figure*}

\begin{table}[ht]
\centering
\small
\caption{\rm MPC + observer -- timing breakdown as a fucntion of the prediction horizon $N$. SDP solver: MOSEK 8 beta; parsing and monomial reduction: Yalmip.}\label{tab:time1}\vspace{0mm}
\begin{tabular}{cccc}
\toprule
& parsing & monomial reduction & SDP solve  \\\midrule
$N$ = $1$ & 5.1\,s & 2.0 \,s & 0.71 \,s \\ \midrule
$N$ = $2$ & 5.6\,s &  2.7  \,s & 1.1 \,s \\ \midrule
$N$ = $3$ & 6.3\,s & 3.1\,s & 2.3 \,s \\ \midrule
$N$ = $4$ & 7.4\,s & 4.2\,s & 4.5 \,s \\ \midrule
$N$ = $5$ & 9.2\,s & 5.9\,s & 11.1 \,s \\ 
 \bottomrule
\end{tabular}
\end{table}

\subsection{Stability of a quadcopter on a given subset}
This example investigates stability of a linearized attitude and vertical velocity model of a quadcopter. The system has seven states (Roll, Pitch and Yaw angles and angular velocities, and velocity in the vertical direction) and four control inputs (the thrusts of the four rotors). The system is controlled by a one-step MPC controller (with perfect state information) which at time~$k$ approximately minimizes the cost $x_k^TQx_k + u_k^TRu_k+x_{k+1}^TPx_{k+1}$, where $Q = I$, $R= 10 I$ and $P$ is the infinite-time LQ matrix associated to the cost matrices $Q$ and $R$, using one step of the projected gradient method~(\ref{eq:gradStep}) subject to the input constraints $\|u \|_\infty \le 1$. This model is open-loop unstable and therefore we investigate closed-loop stability in the region ${\bf X}=[-1,1]^7$ as described in Section~\ref{sec:stabLocal}. The SOS problem~(\ref{opt:sos}) is feasible when seeking a quadratic Lyapunov function using SOS multipliers $\sigma_1$, $\sigma_2$ in equation~(\ref{eq:sos_decrease}) of degree two in $x$ and the polynomial multipliers $p_1$, $p_2$ of degree one in $(x,\theta, \lambda)$. The smallest set of monomials constituting~$\sigma_0$ is chosen automatically by SOSOPT. In~(\ref{eq:sos_nonneg}), we chose all multipliers zero except for $\bar{\sigma}_0$ whose monomials are determined automatically by SOSOPT. Computing the largest $\gamma$ such that $\{x \mid V(x) \le \gamma \}$ is included in $\bf X$ yields $\gamma = 6.37$; this proves that all trajectories starting in $\{x \mid V(x) \le \gamma \}$ stay there and converge to the origin. One closed-loop trajectory of $\|x\|_2$, $V(x)$ and $u$ are depicted in Figure~\ref{fig:quad}; note that this trajectory does not start in $\{x \mid V(x) \le \gamma \}$ but still converges to the origin and the Lyapunov function decreases. The parsing time and monomial reduction carried out by SOSOPT took $2.7$\,s and $16.2$\,s, respectively; the MOSEK solve time was 0.55\,s.

\begin{figure*}[th]
\begin{picture}(140,40)
\put(5,3){\includegraphics[width=40mm]{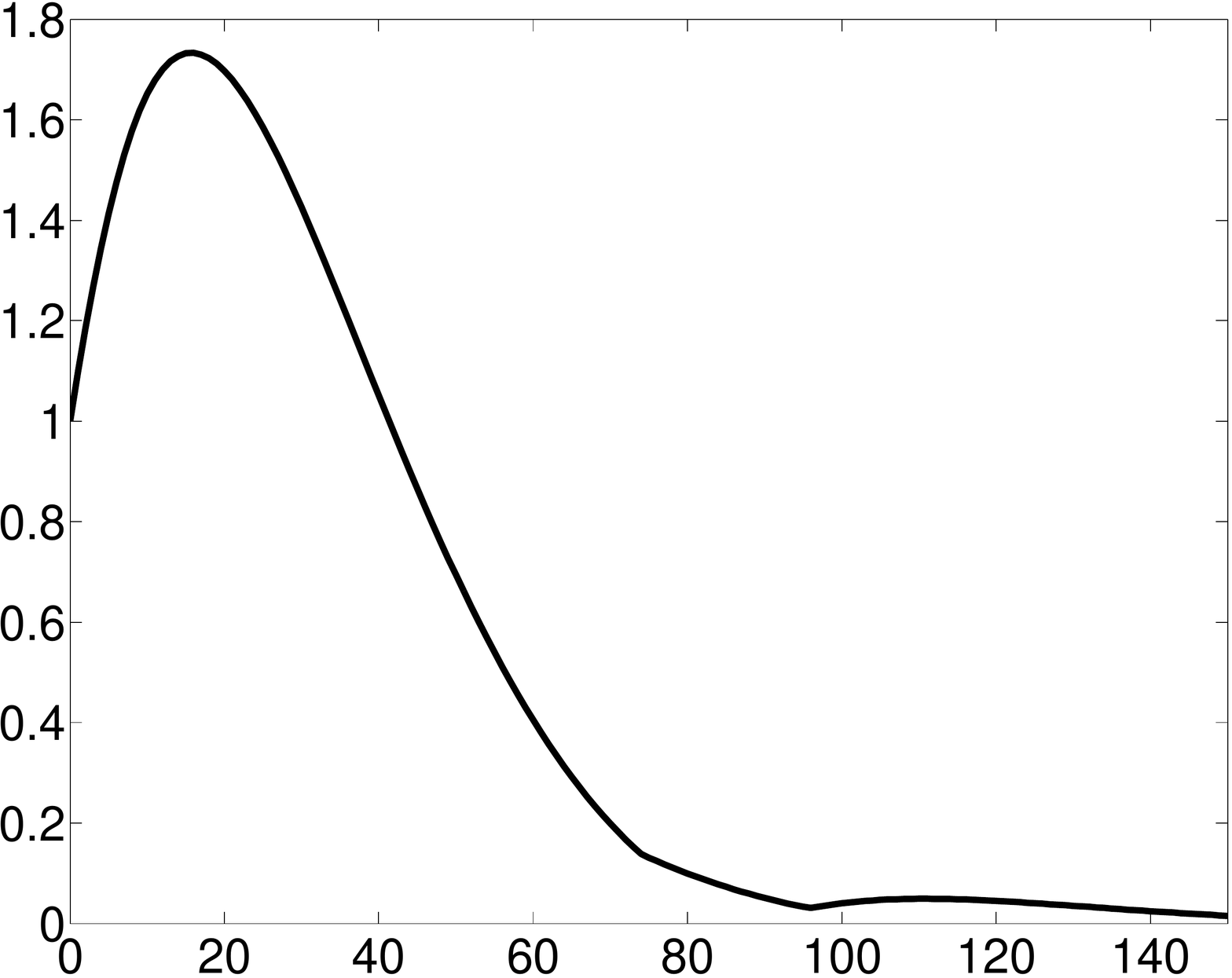}}
\put(55,3){\includegraphics[width=40mm]{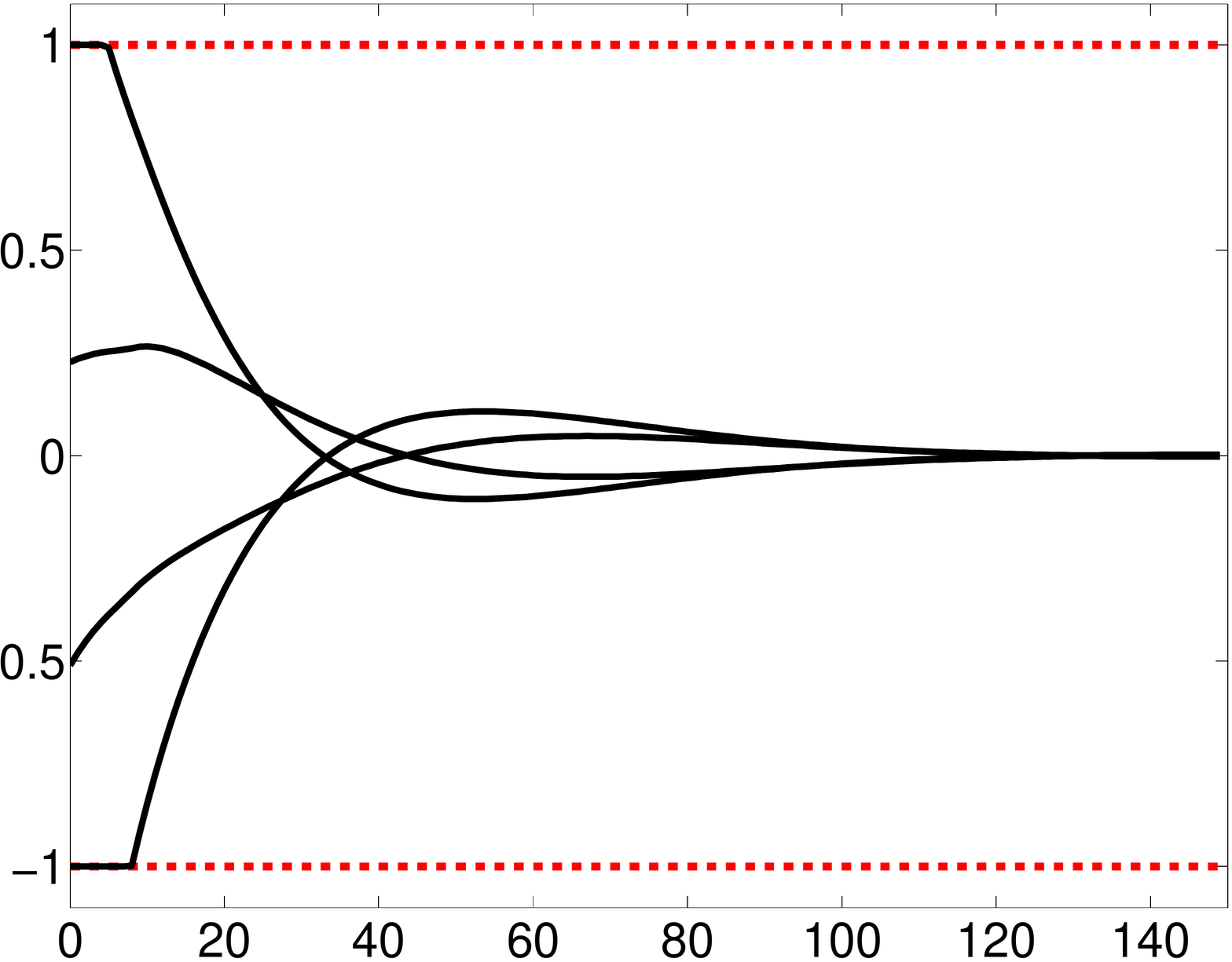}}
\put(105,3){\includegraphics[width=40mm]{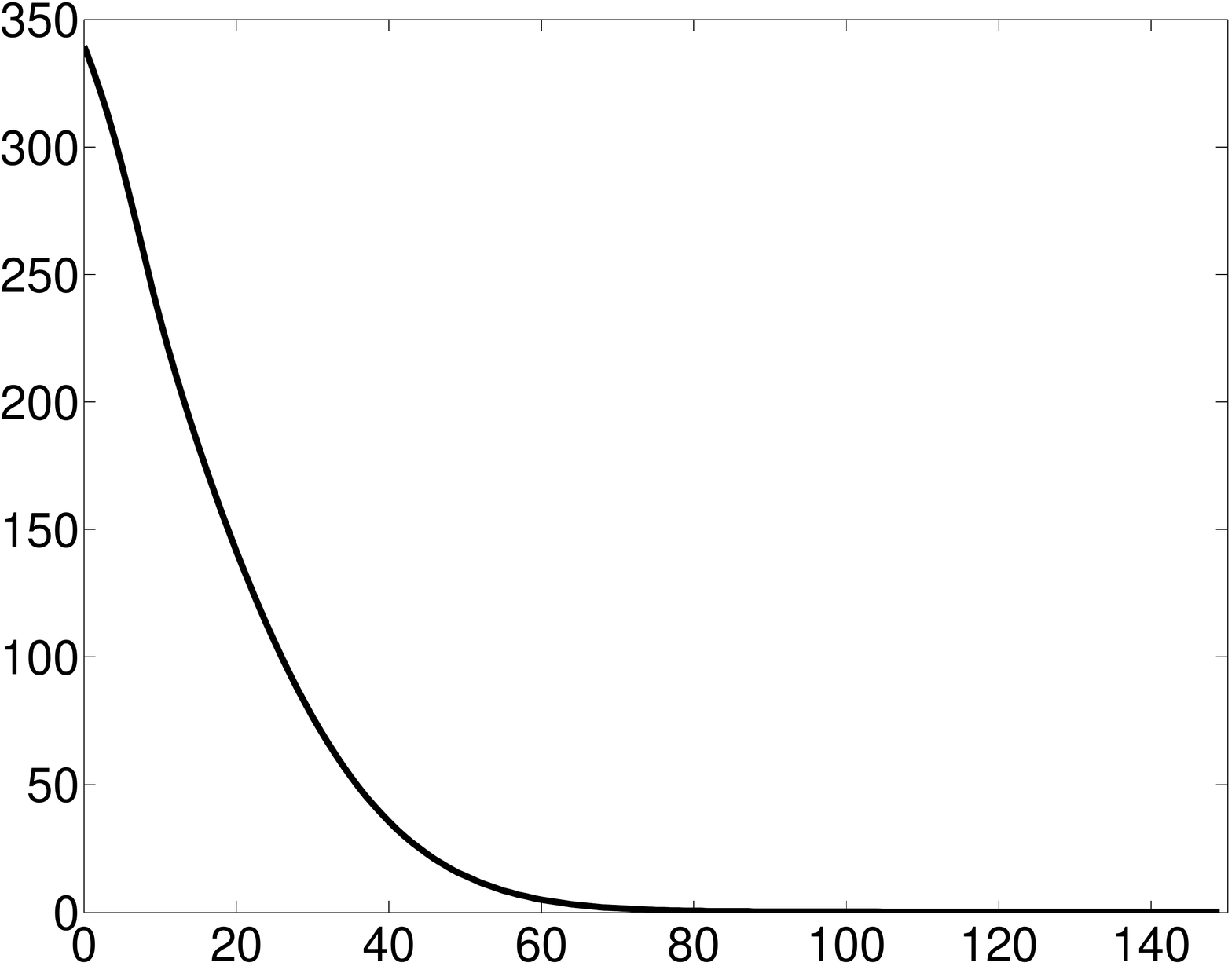}} 


\put(25,0.5){\footnotesize $k$}
\put(75,0.5){\footnotesize $k$}
\put(126,0.5){\footnotesize $k$}

\put(25,10){ $\|x_k\|_\infty$}
\put(65,22){ $u_k$}
\put(115,20){ $V(x_k)$}

\end{picture}
\caption{Local stability of a quadcopter -- trajectories of the norm of the state $\|x_k\|$, the Lyapunov function $V(x_k)$, the control input $u_k$ for initial condition $x_0 = [1 \; 1\; 1\; 1 \;1 \;1\; 1]^\top$.}
\label{fig:quad}
\end{figure*}

\section{Conclusion}
We presented a method for closed-loop analysis of polynomial dynamical systems controlled by an optimization based controller. Provided that all data is polynomial the analysis problem boils down to a semidefinite programming (SDP) problem which can be readily modeled and solved using freely available tools. The limiting factor is the parsing time of the SOS problems, which, however, should be possible to ameliorate by a tailored implementation. Besides the parsing time, the second limiting factor is the size of the resulting SDP problem, especially if tighter performance bounds are required; using first-order-like SDP solvers (e.g., SDPNAL~\cite{sdpnal}), and/or parellel solvers (e.g., SDPARA~\cite{sdpara}) should enable the presented method to scale beyond the reach of interior point methods used in this paper.

\section*{Acknowledgments}
The first author would like to thank Ye Pu for a discussion that initiated the work, and Zlatko Emedji and Mahdieh Sadat Sadabad for providing the uncertain system for the numerical example.

\section*{Appendix}
\subsection*{Proof of Lemma~\ref{lem:ub}}
Let $(x_t,z_t)$ be a solution to~(\ref{eq:difIncl}) for $t \in \{0,1,\ldots\}$ and let $u_t \in \kappa(\mathbf{K}_s)$. Then there exist $\theta_t \in \Rb^{n_\theta}$ and $\lambda_t \in \Rb^{n_\lambda}$ such that
\begin{equation}\label{eq:barKaux}
(x_t,z_t,\theta_t,\lambda_t,x_{t+1},z_{t+1},\theta_{t+1},\lambda_{t+1}) \in \bf \bar{K}
\end{equation}
and
\begin{equation}\label{eq:hatKaux}
(x_t,z_t,\theta_t,\lambda_t) \in \bf \hat{K}
\end{equation}
for all $t \in \{0,1,\ldots, \tau-1\}$, where $\tau := \tau(x_0,z_0)$ is defined in~(\ref{eq:tau}) and
\begin{equation}\label{eq:KcAux}
(x_{\tau},z_{\tau},\theta_{\tau},\lambda_{\tau}) \in \mathbf{\hat{K}_c}.
\end{equation}
Using~(\ref{eq:ub_V}) and (\ref{eq:barKaux}), we conclude that
\[
V(x_t,z_t,\theta_t,\lambda_t) - \alpha V(x_{t+1},z_{t+1},\theta_{t+1},\lambda_{t+1}) - l(x_t,u_t) \ge 0 
\]
for all $t \in \{0,1,\ldots, \tau-1\}$. This implies that
\[
\alpha^{\tau}V(x_{\tau},z_{\tau},\theta_{\tau},\lambda_{\tau}) + \sum_{t=0}^{\tau-1}\alpha^t l(x_t,u_t) \le V(x_0,z_0,\theta_0,\lambda_0) .
\]
Using~(\ref{eq:ub_V_L}) and (\ref{eq:KcAux}) we conclude that
\[
\mathcal{C}(x_0,z_0) = \alpha^{\tau}L + \sum_{t=0}^{\tau-1}\alpha^t l(x_t,u_t) \le V(x_0,z_0,\theta_0,\lambda_0) .
\]
Using~(\ref{eq:ub_barV}) and (\ref{eq:hatKaux}) we have $ V(x_0,z_0,\theta_0,\lambda_0) \le \overline{V}(x_0,z_0)$ and hence $\overline{V}(x_0,z_0) \ge \mathcal{C}(x_0,z_0) $ as desired. \hfill $\blacksquare$

\subsection*{Proof of Lemma~\ref{lem:ub_stoch}}
The proof proceeds along similar lines as the deterministic version by decomposing the probability space according to the values of the stopping time $\tau$. On the probability event $\{\tau = k\}$ we get, by iterating the inequality~(\ref{eq:ub_V_stoch}),
\[
\alpha^{k}V(x_{k},z_{k}) + \E\Big\{ \sum_{t=0}^{k-1}\alpha^t l(x_t,u_t) \mid \tau = k \Big \} \le V(x_0,z_0).
\]
Since $(x_k,z_k)\notin \mathbf{X}$ on $\{\tau = k\}$ we have $V(x_{k},z_{k}) \ge L$ by (\ref{eq:ub_V_L_stoch}) and hence
\[
\alpha^{k}L + \E\{ \sum_{t=0}^{k-1}\alpha^t l(x_t,u_t) \mid \tau = k\} \le V(x_0,z_0)
\]
on $\{\tau = k\}$. Summing over $k$ gives the result.
\subsection*{Proof of Lemma~\ref{lem:rob}}
Let $(x_t,z_t)$ be a solution to~(\ref{eq:difIncl_stoch}) with $(w_t,v_t) \in \mathbf{W}(x_t,z_t)$ for all $t \in \{0,1,\ldots\}$. Then there exist $\theta_t \in \Rb^{n_\theta}$ and $\lambda_t \in \Rb^{n_\lambda}$ such that
\[
(x_t,z_t, \theta_t, \lambda_t, w_t, v_t,x_{t+1} z_{t+1}, \theta_{t+1}, \lambda_{t+1},w_{t+1},v_{t+1})  \in \mathbf{K}_w
\]
and
\[
(x_t,z_t, \theta_t, \lambda_t, w_t, v_t,) \in \hat{\mathbf{K}}_w
\]
for all $t\in \{0,1,\ldots\}$. Hence, by~(\ref{eq:LyapDec_w}) and (\ref{eq:LyapNonNeg_w}),
\begin{align}
 \nonumber &V(x_{t+1},z_{t+1},\theta_{t+1},\lambda_{t+1},w_{t+1},v_{t+1}) \! - \!  V(x_t,z_t,\theta_t,\lambda_t,w_t,v_t)\\   & \hspace{2.1cm} \le - \| \hat{y}_t\|_2^2 + \alpha_w \| w_t\|_2^2 + \alpha_v \| v_t\|_2^2 \label{eq:robAux1}
\end{align}
and
\begin{equation*}
V(x_t,z_t,\theta_t,\lambda_t) \ge 0
\end{equation*}
for all $t\in \{0,1,\ldots\}$. Iterating~(\ref{eq:robAux1}) we obtain
\begin{align*}
&\sum_{t=0}^{T-1} \| \hat{y}_t\|_2^2 \le - V(x_T,z_T,\theta_T,\lambda_T,w_T,v_T) \\ & + V(x_0,z_0,\theta_0,\lambda_0,w_0,v_0)  +\sum_{t=0}^{T-1} \alpha_w \| w_t\|_2^2+ \alpha_v\| v_t\|_2^2 \\ &\le V(x_0,z_0,\theta_0,\lambda_0,w_0,v_0)  + \sum_{t=0}^{T-1} \alpha_w \| w_t\|_2^2+  \alpha_v\| v_t\|_2^2,
\end{align*}
where we have used the fact that $V(x_T,z_T,\theta_T,\lambda_T) \ge 0$ in the second inequality. Letting $T$ tend to infinity gives the result. \hfill $\blacksquare$

\bibliographystyle{abbrv}
\bibliography{./References}

\begin{thebibliography}{10}

\bibitem{boyd_LMI}
S.~Boyd, L.~E. Ghaoui, E.~Feron, and V.~Balakrishnan.
\newblock {\em Linear Matrix Inequalities in System and Control Theory}.
\newblock Society for Industrial and Applied Mathematics (SIAM), 1994.

\bibitem{gabay1976_admm}
D.~Gabay and B.~Mercier.
\newblock A dual algorithm for the solution of nonlinear variational problems
  via finite element approximation.
\newblock {\em Computers \& Mathematics with Applications}, 2(1):17--40, 1976.

\bibitem{goldstein2014fast}
T.~Goldstein, B.~O'Donoghue, S.~Setzer, and R.~Baraniuk.
\newblock Fast alternating direction optimization methods.
\newblock {\em SIAM Journal on Imaging Sciences}, 7(3):1588--1623, 2014.

\bibitem{jeanNonconvex}
J.~H. Hours and C.~N. Jones.
\newblock A parametric multi-convex splitting technique with application to
  real-time {NMPC}.
\newblock In {\em 53st IEEE Conference on Decision and Control (CDC 2014)},
  2014.

\bibitem{korda_ACC_verif}
M.~Korda and C.~N. Jones.
\newblock Certification of fixed computation time first-order
  optimization-based controllers for a class of nonlinear dynamical systems.
\newblock In {\em American Control Conference}, 2013.

\bibitem{lasserreBook}
J.~B. Lasserre.
\newblock {\em Moments, Positive Polynomials and Their Applications,}.
\newblock Imperial College Press, first edition, 2009.

\bibitem{lasserre_convexRepresentation}
J.~B. Lasserre.
\newblock On representations of the feasible set in convex optimization.
\newblock {\em Optimization Letters}, 4(1):1--5, 2010.

\bibitem{yalmip}
J.~L{\" o}fberg.
\newblock Yalmip : A toolbox for modeling and optimization in {MATLAB}.
\newblock In {\em Proceedings of the CACSD Conference}, Taipei, Taiwan, 2004.

\bibitem{mosek}
{MOSEK ApS}.
\newblock {\em The MOSEK optimization toolbox for MATLAB manual}, 2016.

\bibitem{nesterov_book}
Y.~Nesterov.
\newblock {\em Introductory lectures on convex optimization}, volume~87.
\newblock Springer, 2004.

\bibitem{sostools}
A.~Papachristodoulou, J.~Anderson, S.~P. G.~Valmorbida, P.~Seiler, and P.~A.
  Parrilo.
\newblock {\em {SOSTOOLS}: Sum of squares optimization toolbox for {MATLAB}},
  2013.

\bibitem{parrilo}
P.~A. Parrilo.
\newblock Semidefinite programming relaxations for semialgebraic problems.
\newblock {\em Mathematical programming}, 96(2):293--320, 2003.

\bibitem{facialReduction}
F.~Permenter and P.~Parrilo.
\newblock Partial facial reduction: Simplified, equivalent semidefinite
  programs via approximations of the positive semidefinite cone.
\newblock {\em Arxiv preprint {arXiv:0905.3447}}, 2014.

\bibitem{peterson1973}
D.~W. Peterson.
\newblock A review of constraint qualifications in finite-dimensional spaces.
\newblock {\em SIAM Review}, 15(3):639--654, 1973.

\bibitem{sedumi}
I.~P{\'o}lik, T.~Terlaky, and Y.~Zinchenko.
\newblock Sedumi: a package for conic optimization.
\newblock In {\em IMA workshop on Optimization and Control}, 2007.

\bibitem{primbs}
J.~A. Primbs.
\newblock The analysis of optimization based controllers.
\newblock {\em Automatica}, 37(6):933--938, 2001.

\bibitem{sosopt}
P.~Seiler.
\newblock {SOSOPT}: {A} toolbox for polynomial optimization.
\newblock {\em University of Minnesota}, 2010.

\bibitem{sturmfels}
B.~Sturmfels.
\newblock {Polynomial equations and convex polytopes. American Mathematical
  Monthly}.
\newblock {\em {American Mathematical Monthly}}, 105(10):907--922, 1998.

\bibitem{tseng_ama}
P.~Tseng.
\newblock Applications of a splitting algorithm to decomposition in convex
  programming and variational inequalities.
\newblock {\em SIAM Journal on Control and Optimization}, 29(1):119--138, 1991.

\bibitem{sdpara}
M.~Yamashita, K.~Fujisawa, and M.~Kojima.
\newblock {SDPARA}: Semidefinite programming algorithm parallel version.
\newblock {\em Parallel Computing}, 29(8):1053--1067, 2003.

\bibitem{sdpnal}
X.-Y. Zhao, D.~Sun, and K.-C. Toh.
\newblock A {N}ewton-{CG} augmented lagrangian method for semidefinite
  programming.
\newblock {\em {SIAM} Journal on Optimization}, 20(4):1737--1765, 2010.

\end{thebibliography}

\end{document}